\documentclass[11pt,reqno]{amsart}
\usepackage{t1enc}
\usepackage[latin1]{inputenc}
\usepackage[english]{babel}
\usepackage{amsmath,amsthm}
\usepackage{amssymb}
\usepackage{amsfonts}
\usepackage{latexsym}
\usepackage[dvips]{graphicx}
\usepackage{graphicx}
\usepackage{float}
\usepackage[natural]{xcolor}
\usepackage{algorithm}
\usepackage{algorithmic}
\usepackage{enumerate}
\usepackage{appendix}
\usepackage{multirow}
\usepackage{xcolor}
\usepackage[colorlinks,linkcolor=blue]{hyperref}
\usepackage{lineno}
\usepackage{setspace}
\usepackage{comment}
\usepackage{fullpage}
\usepackage[shortlabels]{enumitem}

\usepackage{listings}
\newtheorem{theorem}{Theorem}\numberwithin{theorem}{section}
\newtheorem{lemma}[theorem]{Lemma}\numberwithin{theorem}{section}
\newtheorem{proposition}[theorem]{Proposition}
\numberwithin{theorem}{section}

\newtheorem{example}[theorem]{Example}
\newtheorem{definition}[theorem]{Definition}

\newtheorem{claim}[theorem]{Claim}

\def\int{\textrm{int}}

\begin{document}

\onehalfspacing
\title{On the number of factorable induced subgraphs}

\author{Jie Han}
\address{JH, BW and JZ. School of Mathematics and Statistics, Beijing Institute of Technology, China\\
Email: \texttt{(JH) han.jie@bit.edu.cn, (BW) bin.wang@bit.edu.cn, (JZ) jingwen.zhao@bit.edu.cn.}}
\author{Bin Wang}
\author{Jingwen Zhao}

\begin{abstract}
Let $F$ be an $r$-vertex graph. 
In this paper, we study the $F$-factor problem in random induced subgraphs of dense graphs.
We show that for any $r$-vertex graph $F$ and $\gamma>0$, if $H$ is an $n$-vertex graph with minimum degree at least $(1-1/\chi_{cr}(F)+\gamma)n$, then for every fixed $p \in (0,1)$, the random induced subgraph $H[p]$ contains an $F$-factor with probability at least $1/(rq)-o_n(1)$, where $q\in \mathbb{N}$ is the order of certain coset group defined from $H$. 
The probability is asymptotically best possible  for infinitely many $F$ and $H$ and yields that a $1/(rq)-o_n(1)$ proportion of the subsets of $H$ induce $F$-factors, interestingly, regardless of whether $H$ itself admits an $F$-factor.
Similar results are obtained for perfect matchings in hypergraphs under minimum degree conditions.
Our proof combines concentration inequalities, lattice point counting in $\mathbb{Z}^d$ and structural theorems for $F$-factors in dense (hyper)graphs.
\end{abstract}

\maketitle

\section{Introduction}
\subsection{Factors in (hyper)graphs}
Given a $k$-uniform hypergraph $F$, an $F$-packing in $k$-graph $H$ is a collection of vertex-disjoint copies of $F$ in $H$. 
An $F$-packing is called perfect if it covers all vertices of $H$. 
A perfect $F$-packing is also called an $F$-factor. 
A classical line of research in extremal graph theory is to determine Dirac-type conditions for spanning substructures in hypergraphs. 
If $F$ has a component with at least 3 vertices then the question whether $H$ has an $F$-factor is difficult from both structural and
algorithmic points of view: Tutte's theorem characterizes those graphs which
have an $F$-factor if $F$ is an edge but for other graphs $F$ no such characterization exists. 
Moreover, Hell and Kirkpatrick \cite{MR707416} showed that the
decision problem whether a graph $H$ has an $F$-factor is NP-complete
if and only if $F$ has a component with at least 3 vertices.
This motivates the research for simple sufficient conditions ensuring the
existence of an $F$-factor.
The following classical result of Hajnal and Szemer\'{e}di characterizes the minimum degree
that ensures a graph contains a $K_r$-factor.
\begin{theorem}[Hajnal and Szemer\'{e}di~\cite{MR297607}]
     Every graph $G$ whose order $n$ is divisible by $r$ and whose minimum degree satisfies $\delta(G)\geq (1-1/r)n $ contains a $K_r$-factor.
\end{theorem}
Koml\'os, S\'ark\"{o}zy and Szemer\'edi \cite{KSS2001factor} generalized the result to arbitrary $F$-factors.
They proved that for every graph $F$ there exists a constant $C=C(F)$ such that every graph $G$ whose order $n$ is divisible by $|V(F)|$ and whose minimum degree is at least $(1-1/\chi(F))n+C$ contains an $F$-factor.
This confirmed a conjecture of Alon and Yuster \cite{alonyuster}, who had obtained the above result with an additional error term of $\varepsilon n$ in the minimum degree
condition. 

As observed in \cite{alonyuster}, there are graphs $F$ for which the above constant $C$ cannot be omitted completely. 
However, K\"{u}hn and Osthus \cite{kuhnothus} proved that for some graphs the minimum degree threshold can be improved significantly by replacing the chromatic number $\chi(F)$ with a refined parameter known as the critical chromatic number $\chi_{cr}(F)$.
The \emph{critical chromatic number} $\chi_{cr}(F)$ of a graph $F$ was introduced by Koml\'{o}s \cite{komlos2000almost}, which
is defined as $\chi_{cr}(F)=\frac{(\chi(F)-1)|V(F)|}{|V(F)|-\sigma(F)},$
where $\sigma(F)$ denotes the minimum size of the smallest colour class in a $\chi(F)$-colouring of $F$. 
Note that $\chi(F)-1 < \chi_{cr}(F) \le \chi(F),$ and equality holds if and only if, for every $\chi(F)$-colouring of $F$, all colour classes have the same size. 
Furthermore, Koml\'{o}s \cite{komlos2000almost} proved that if an $n$-vertex graph $G$ satisfies the degree condition $\delta(G)\ge \left(1-1/\chi_{cr}(F)\right)n$ for sufficiently large $n$, then $G$ contains an almost $F$-factor.
More recently, Han and Treglown \cite{han2020complexity} established a polynomial algorithm which determines whether $G$ with minimum degree condition $\delta(G)\ge \left(1-1/\chi_{cr}(F)\right)n$
contains an $F$-factor.


In hypergraphs, a major milestone in this area is the work of R\"{o}dl, Ruci\'{n}ski and Szemer\'{e}di \cite{RSS2009perfect}, who determined the minimum codegree threshold for perfect matchings in large uniform hypergraphs and introduced the absorbing method into the subject. 
Let $k\geq2$, a $k$-graph $H=(V,E)$ consists of a vertex set of order $n$ and an edge set $E\subseteq\binom{V}{k}$.
For any $(k-1)$-subset $S$ of $V$, the \emph{codegree} of $S$, denoted by $\deg_H(S)$ is the number of edges containing $S$.
The \emph{minimum} \emph{codegree} $\delta_{k-1}(H)$ is the minimum of $\deg_H(S)$ over all $(k-1)$-subsets $S$ of $V$.
They determined
the minimum codegree threshold that ensures a perfect matching in a $k$-graph on
$n$ vertices for all $k\ge3$ and sufficiently large $n\in k\mathbb{N}$. The threshold is $n/2-k+C$,
where $C\in\{3/2,2,5/2,3\}$ depends on the values of $n$ and $k$. 
They also showed that the condition $\delta_{k-1}(H) \ge n/k+O(\log n)$ is sufficient to guarantee a matching covering all but at most $k$ vertices of $H$, i.e. one edge away from a perfect matching.
They conjectured that $\delta_{k-1}(H)\ge n/k$ suffices for this, which was recently proved by Han \cite{Han2015near}. 
Let $H$ be a $k$-graph on $n$ vertices, with minimum codegree at least $n/k+cn$ for some fixed $c>0$, 
Keevash, Knox and Mycroft \cite{KKM2015matchingpoly}
constructed a polynomial-time algorithm which finds either a perfect matching in $H$ or a certificate that none exists. 
This essentially solves a problem of Karpi\'{n}ski, Ruci\'{n}ski and 
Szyma\'{n}ska \cite{KRZ2010matching}.


More generally, Han and Treglown \cite{han2020complexity} established a general lattice-based criterion for perfect matchings and factors in dense hypergraphs, and used it to obtain algorithmic and structural consequences. Their result shows, roughly speaking, that once an appropriate minimum $\ell$-degree condition and a suitable structural information are available, the existence of a factor is governed by the solubility of an associated lattice system (see details in Section \ref{sec:generalfra}).

\subsection{Factors in random induced subgraphs}
An important topic in graph theory is studying when certain classical theorems
hold in a ``robust'' or ``resilient'' way, according to various possible interpretations of these terms. While many classical results can be interpreted as part of this direction, it was first highlighted as a topic for systematic study by Sudakov and Vu \cite{sudakov2008robust}. 

Let $H$ be a $k$-graph with a property $\mathcal{P}$.
On the one hand, one can sample each edge of $H$ uniformly at random with probability $p$ to obtain a binomial subhypergraph of $H$.
A central topic is to study the threshold for the property $\mathcal{P}$ of this binomial subhypergraph of $H$.

On the other hand, one can sample vertices rather than edges at random.
For a finite set $V$ and $p\in [0,1]$, let $V_p\subseteq V$ be obtained by including each element of $V$ independently with probability $p$.
Given a $k$-graph $H=(V, E)$ and $p\in [0,1]$, let $H[p]:=H[V_p]$ be a random induced subgraph of $H$.
We focus on this model and investigate the probability that the induced subgraph on such a random vertex 
set $H[p]$ still has graph property $\mathcal{P}$.
The study of induced random subgraphs has recently attracted considerable attention.
 In 1996, in his last paper, Erd\H{o}s \cite{MR1684620} asked the following question that he formulated together
with Faudree: is there a positive $c$ such that any $(n+1)$-regular graph $G$ on $2n$ vertices contains at least $c2^{2n}$ distinct vertex-subsets $S$ that are cyclic, meaning that there is a cycle in $G$ using precisely the vertices in $S$. 
This was recently solved for large $n$ by Dragani\'{c}, Keevash and M\"{u}yesser \cite{MR4935989}, who showed that a uniformly random vertex subset of an $(n + 1)$-regular $2n$-vertex graph induces a Hamiltonian graph with probability at least $1/2$. 
Throughout this paper, we write $o_m(1)$ to denote a quantity that tends to 0 as $m$ tends to $\infty$. 
When $m$ is clear from the context, we omit the subscript.
\begin{theorem}[Dragani\'{c}, Keevash and M\"{u}yesser~\cite{MR4935989}]
    Any $(n+1)$-regular graph $G$ on $2n$ vertices has $\mathbb{P}[G[1/2]\text{\ is\ Hamiltonian}]>1/2-o_n(1)$.
\end{theorem}

The bound $1/2$ is tight, as shown by the example of the complete bipartite graph $K_{n-1,n+1}$ with a 2-factor added to the larger side. 
Writing $\operatorname{Cyc}(G)$ for the number of cyclic subsets of $G$, Liu, Niu, Wang and Yan \cite{LiuNiuWangYan2026} extended this line of research to regular graphs below the Dirac threshold.
They proved that, for every $\varepsilon,\xi>0$ and all sufficiently large $n$, if $G$ is an $n$-vertex $d$-regular graph with $\varepsilon n\le d<n/2$ and $q:=\lfloor n/(d+1)\rfloor$, then $\operatorname{Cyc}(G)\ge (q-\xi)2^{n/q}$.
They also determined the optimal exponential rate at the Dirac boundary, proving that every $n$-vertex $n/2$-regular graph satisfies $\operatorname{Cyc}(G)\ge 2^{(1-o(1))n}$.
Very recently, Hunter, Liu, Milojevi\'{c} and Sudakov \cite{MR5041387} investigated a tournament analogue under a minimum semi-degree condition.
Dragani\'{c}, Keevash and M\"{u}yesser \cite{MR4935989} conjectured that for any $r\ge2$, there is some constant $c>0$ so that if $G$ is an $((r-1)n+1)$-regular graph on $rn$ vertices, then at least $c2^{rn}$ subsets of $V(G)$ induce a $K_r$-factor.
This conjecture was resolved by Sun, Wei and Yang \cite{Yang2025fator} for sufficiently large $n$.
\begin{theorem}[Sun, Wei and Yang \cite{Yang2025fator}]\label{thm:swy}
    For any $r\ge2$, there is a constant $c>0$ such that the following holds for sufficiently large $n$.
    Let $G$ be an $((r-1)n+1)$-regular graph on $rn$ vertices.
    Then at least $c2^{rn}$ subsets of $V(G)$ induce a $K_r$-factor.
\end{theorem}

In fact, they proved that the constant $c$ can be taken as $\frac{1}{(40r^2)^r}$.
They also conclude that for any probability $p\in(0,1)$, $G[p]$ contains a $K_r$-factor with probability at least $(\frac{p^2}{20r^2})^r$.

\subsection{Main results}
In this paper, we study the $F$-factor problem in random induced subgraphs for general $F$.
First, we give a simple proposition by concentration of minimum degrees.
\begin{proposition}  
\label{prop:simple}
Let \(r\ge 2\) be an integer and \(\gamma,\varepsilon>0\). 
Then there exists $C,n_0\in \mathbb{N}$ such that the following holds.
Let \(F\) be an \(r\)-vertex graph 
and \(H\) be a graph on $n\ge n_0$ vertices such that \(\delta(H)\ge (1-1/\chi(F)+\gamma)n\).
    If \(pn\ge C\log n\) and $p(1-p)n\ge C$, then \(\mathbb{P}\left[H[p]\text{ contains an $F$-factor}\right]\ge {1}/{r}-\varepsilon\). 
    \qed
\end{proposition}

In fact, by Chernoff's inequality, $H[p]$ ``inherits'' the minimum degree condition of $H$ with probability $1-o(1)$. Thus, by the Alon--Yuster Theorem \cite{alonyuster}, $H[p]$ has an $F$-factor if and only if $|V(H[p])|$ is divisible by $r$, which happens with probability $1/r-o(1)$, again by concentration.
Combining these two events give the proposition.

Our main result explores the situation when the minimum degree is relaxed significantly.

\begin{theorem}
\label{thm:unbalancedfactor}
Let \(r\ge k\ge 2\) be integers and \(\gamma,\varepsilon>0\).
Then there exist $n_0,C'\in \mathbb{N}$ such that the following holds.
Let \(F\) be an \(r\)-vertex \(k\)-chromatic graph and \(H\) be a graph on $n\ge n_0$ vertices such that \(\delta(H)\ge (1-1/\chi_{cr}(F)+\gamma)n\).
    Let \(h := 2^{r^{k-1}}r-2\).
    If \(p\ge C'(\log{n}/n)^{1/h}\) and \(p(1-p)\ge C'n^{-1/2}\), 
    then there exists $q\le (2r-1)^r$ such that
    \(\mathbb{P}\left[H[p]\text{ contains an $F$-factor}\right]\ge {1}/{(rq)}-\varepsilon\).
\end{theorem} 

A natural consequence of Theorem \ref{thm:unbalancedfactor}  (with $p=1/2$) is that at least $(1/(rq)-o(1))2^{n}$ subsets of $H$ induce $F$-factors, regardless of whether $H$ itself admits an $F$-factor.
Moreover, Theorem \ref{thm:unbalancedfactor} is (asymptotically) best possible in multiple senses. 
\begin{itemize}
    \item First, the minimum degree assumption cannot be weakened substantially.
    Indeed, if we relax the minimum degree condition further significantly then the host graph $H$ might be contained in the so-called \emph{space barrier}, resulting that the probability of $H[p]$ containing an $F$-factor tends to 0 -- see the construction below.
    
\medskip
\noindent
\emph{Construction 1.}
Let $\gamma>0$, $F$ be an $r$-vertex graph and $n\in \mathbb{N}$ be sufficiently large. 
Let \(H_0\) be the complete \(\chi(F)\)-partite graph on \(n\) vertices with one vertex class of size \((\sigma(F)/r-\gamma)n\), denoted by \(A\), and all other vertex classes as equal in size as possible. 
It is easy to see that $\delta(H_0) = n-\frac{n-|A|}{\chi(F)-1}= (1-\frac{1}{\chi_{cr}(F)}-\frac{\gamma}{\chi(F)-1})n$ and for fixed $p\in (0,1)$ with probability $1-o_n(1)$, a $p$-random subset $V_p$ of $V(H_0)$ satisfies that $|V_p\cap A|\le(\sigma(F)/r-\gamma/2)np$ and $|V_p|\ge (1-\gamma^2)np > r|V_p\cap A|/\sigma(F)$.
Therefore, as every copy of \(F\) uses at least \(\sigma(F)\) vertices from \(A\), any maximum \(F\)-packing in $H_0[p]=H_0[V_p]$ has size $ |V_p\cap A|/\sigma(F)< |V_p|/r$, and thus is not perfect.
Therefore, \(\mathbb{P}\left[H_0[p]\text{ contains an \(F\)-factor}\right]=o_n(1)\).

\item Second, the quantity $q$ is the order of certain coset group defined via the structural information of $H$, and for given $q$, there are infinitely many graphs $H$ with $q=q(H,F)$ and \(\mathbb{P}\left[H[p]\text{ contains an $F$-factor}\right]= {1}/{rq}-o_n(1)\). 
See Example \ref{ex:infinite} in Section \ref{sec:notation}.
The winning probability is asymptotically best possible for infinitely many $H$, and we shall explain this in Section \ref{sec:win}.

\item At last, the bound on $p$ relies on our reachability--absorption approach and is probably not optimal. Nevertheless, we include a best possible bound under our approach in case it is useful elsewhere.
\end{itemize}

Indeed, we believe that when one raises the minimum degree of $H$ from $(1-1/\chi_{cr}(F)+\gamma)n$ to $(1-1/\chi(F)+\gamma)n$, the quantity $q$ should monotonically decrease to 1, leaving \(\mathbb{P}\left[H[p]\text{ contains an $F$-factor}\right]\allowbreak= {1}/{r}-o_n(1)\).


We also give a similar result for perfect matchings in hypergraphs.
Given a $k$-graph $H=(V,E)$, a fractional matching in $H$ is a function
$\omega:E\to[0,1]$ such that for each $v\in V$ we have that
$\sum_{e\ni v}\omega(e)\le 1$. Then $\sum_{e\in E}\omega(e)$ is the size of
$\omega$. If the size of $\omega$ in $H$ is $n/k$, then we say that $\omega$
is a \emph{perfect fractional matching}. Given $k,\ell\in\mathbb{N}$ such that
$\ell\le k-1$, define $c_{k,\ell}^{*}$ to be the smallest number $c$ such that
every $k$-graph $H$ on $n$ vertices with
$\delta_{\ell}(H)\ge(c+o_n(1))\binom{n-\ell}{k-\ell}$ contains a perfect
fractional matching. 
Alon et al.~\cite{alon2012large} conjectured that for all $1\le \ell <k$, $c_{k,\ell}^*=1-(1-1/k)^{k-\ell}$ and so far it is verified for $\ell \ge 0.4k$ by Frankl and Kupavskii~\cite{frankl2022erdHos}.
Very recently, Fu et al.~\cite{fu2026sharp} announced a proof of a conjecture of Feige, which, together with a work of Ferber and Jain~\cite{ferber2019uniformity}, implies the conjecture of Alon et al., that is, for all $1\le \ell <k$, $c_{k,\ell}^*=1-(1-1/k)^{k-\ell}$.

\begin{theorem}\label{thm:ell_degree}
Let \(k\ge 3\), \(\ell\in[k-1]\) be integers, and let \(\gamma,\varepsilon>0\). 
    Then there exist \(n_0,C'\in\mathbb N\) such that the following holds. 
    Let \(H\) be a \(k\)-graph on $n\ge n_0$ vertices satisfying \(\delta_{\ell}(H)\ge(c^*_{k,\ell}+\gamma)\binom{n-\ell}{k-\ell}\). 
    There exists an integer $q\le (2k+1)^k$ such that the following holds.
    Let \(h:=\max\{2k,2^{\lfloor1/c^*_{k,\ell}\rfloor}k-2\}\). 
    If \(p\ge C'(\log{n}/n)^{1/h}\) and \(p(1-p)\ge C'n^{-1/2}\), then \(\mathbb{P}\left[H[p]\text{ contains a perfect matching}\right]\ge\frac{1}{kq}-\varepsilon\). 
\end{theorem}

Similar to Theorem \ref{thm:unbalancedfactor}, both the minimum degree conditions and the winning probability in Theorem \ref{thm:ell_degree} are asymptotically best possible (see Construction 2).
Moreover, similar to the previous case, $q$ will be taken as the order of certain coset group defined via  $H$.

In the codegree setting, a stronger probability bound can be obtained by exploiting the more precise structural information available under the degree condition \(\delta_{k-1}(H)\ge (1/s+\gamma)n\).

\begin{theorem}\label{thm:perfectmatching}
Let $k\ge3$ be an integer and let $\gamma,\varepsilon>0$.
Then there exist $n_0,C'\in\mathbb{N}$ such that the following holds.
Let \(H\) be a \(k\)-graph on $n\ge n_0$ vertices such that \(\delta_{k-1}(H)\ge (1/s+\gamma)n\) for integer $s\in[2,k]$. 
Let \(h:=\max\{2k,2^{s-1}k-2\}\). If \(p\ge C'(\log{n}/n)^{1/h}\) and \(p(1-p)\ge C'n^{-1/2}\), then \(\mathbb{P}\left[H[p]\text{ contains a perfect matching}\right]\ge\frac{1}{k(s-1)}-\varepsilon\).
\end{theorem}

Theorem~\ref{thm:perfectmatching} implies that roughly a $1/k(s-1)$-proportion of the induced subgraphs of $H$ have a perfect matching, and the probability $1/k(s-1)-o_n(1)$ is asymptotically best possible, see the final section.
Indeed, this matches our instinct above for $F$-factors -- when the minimum codegree takes form $(1/s+\gamma)n$ and $s$ decreases from $k$ to $2$, the quantity $q=q(H)$ as in Theorem~\ref{thm:unbalancedfactor} takes value $q\le s-1$ (when $s=2$ this is the sharp minimum codegree threshold forcing perfect matching determined by R\"{o}dl, Ruci\'{n}ski and Szemer\'{e}di \cite{RSS2009perfect}).

Moreover, when the minimum codegree is significantly less than $n/k$, the following construction (space barrier) shows that the result of Theorem~\ref{thm:perfectmatching} does not hold anymore. 

\medskip
\noindent
\emph{Construction 2.}
Take $k\ge 3$, $\gamma>0$ and $n$ be sufficiently large.
Let $H_0=(V, E)$ be an $n$-vertex $k$-graph with a vertex partition $V=A\cup B$ such that $|A|=(1-\gamma)n/k$ and $E$ consists of all $k$-subsets of $V$ intersecting $A$.
It is easy to see that $\delta_{k-1}(H_0) = |A|=(1-\gamma)n/k$ and 
\(\delta_{\ell}(H_0)\ge \binom{n-\ell}{k-\ell}-\binom{|B|-\ell}{k-\ell}\ge\left(c^*_{k,\ell}-\kappa\gamma\right)\binom{n-\ell}{k-\ell}\) for some constant \(\kappa=\kappa(k,\ell)>0\). 
Moreover, for fixed $p\in (0,1)$ with probability $1-o_n(1)$, a $p$-random subset $V_p$ of $V$ satisfies that $|V_p\cap A|\le (1-\gamma/2)np/k$ and $|V_p|\ge (1-\gamma^2)np > k|V_p\cap A|$.
Therefore, as all edges of $H_0[p]=H_0[V_p]$ intersect $V_p\cap A$, any maximum matching in $H_0[p]$ has size $|V_p\cap A|< |V_p|/k$, and thus is not perfect.
Therefore, \(\mathbb{P}\left[H_0[p]\text{ contains a perfect matching}\right]=o_n(1)\). 


\medskip

The rest of the paper is organized as follows. In Section~\ref{sec:notation} we introduce notation and recall the concentration inequalities and lattice point counting tools used later. 
In Section~\ref{sec:generalfra}, we present our main structural theorem, which establishes a general framework for dealing with $F$-factor problems in random induced subgraphs.
We also give some applications of this framework.
In Section~\ref{sec:proofofmain}, we give a proof of the main structural theorem. 
Finally, in Section~\ref{sec:win} we present (infinitely many) examples showing that the winning probabilities in our main results are asymptotically best possible.


\section{Notation and Preliminaries}\label{sec:notation}
\subsection{Notation} 
Given an \(n\)-vertex \(k\)-graph $H$ and integer \(d\ge 0\), let \(\mathcal{P}=\{V_0,V_1,\dots,V_d\}\) be a partition of \(V(H)\). 
Throughout this paper, every partition has an implicit ordering of its parts.  
For any vector $\textbf{v} \in \mathbb{Z}^d$, $\textbf{v}|_i$ denotes the $i$-th coordinate of $\textbf{v}$, and define $|\textbf{v}| := \sum_{i=1}^d \textbf{v}|_i$.  
We say that $\textbf{v} \in \mathbb{Z}^d$ is an \emph{$r$-vector} if it has non-negative coordinates and satisfies $|\textbf{v}| =r$.

    \begin{definition}
    [Index vector and lattice]
        Let \(F\) be an \(r\)-vertex \(k\)-graph. 
        The \emph{index vector} \(\mathbf{i}_{\mathcal{P}}(S)\in \mathbb{Z}^d \) of a subset \(S\subseteq V\) with respect to \(\mathcal{P}\) is the vector whose coordinates are the size of intersection of \(S\) with each part of \(\mathcal{P}\) except \(V_0\), namely, \(\mathbf{i}_\mathcal{P}(S)|_i = |S \cap V_i|\) for \(i \in [d]\). 
        Then for any \(\mu >0\),
        \begin{enumerate}[label=(\arabic*)]
            \item \(I_{\mathcal{P},F}^{\mu}(H)\) denotes all r-vectors  \(\mathbf{i}\in \mathbb{Z}^d\) such that \(H\) contains at least \(\mu n^r\) copies of \(F\) with index vector \(\mathbf{i}\); such vectors are said to be \emph{\(\mu\)-robust}.
            \item \(L_{\mathcal{P},F}^{\mu}(H)\) denotes the lattice (that is, the additive subgroup) in \(\mathbb{Z}^d\) generated by \(I_{\mathcal{P},F}^{\mu}(H)\).
        \end{enumerate}
    \end{definition}
In the case of perfect matchings (i.e. when $F$ is an edge), we write $I_{\mathcal{P}}^\mu(H)$ and $L_{\mathcal{P}}^{\mu}(H)$ for $I_{\mathcal{P},F}^\mu(H)$ and $L_{\mathcal{P},F}^{\mu}(H)$ respectively.

Let $q \in \mathbb{N}$. 
A (possibly empty) $F$-tiling $M$ in $H$ of size at most $q$ is called a \emph{$q$-solution} for $(\mathcal{P}, L_{\mathcal{P},F}^{\mu}(H))$ if the index vector of the uncovered vertices satisfies $\textbf{i}_{\mathcal{P}}(V(H) \setminus V(M)) \in L_{\mathcal{P},F}^{\mu}(H)$. 
We say that $(\mathcal{P}, L_{\mathcal{P},F}^{\mu}(H))$ is \emph{$q$-soluble} if such a $q$-solution exists.

For a partition $\mathcal{P}$ with $d$ parts, 
let $L_{\max}^d$ be the lattice generated by all $r$-vectors, i.e.,
\[
L_{\max}^d := \{ \textbf{v} \in \mathbb{Z}^d : r \mid |\textbf{v}| \}.
\]
Now suppose $L\subseteq L_{\max}^{d}$ is a lattice in $\mathbb{Z}^d$. 
The \emph{coset group} of $(\mathcal{P}, L)$ is defined as $Q = Q(\mathcal{P}, L) := L_{\max}^d/L$. 

{In our proofs, we shall define $q=|Q(\mathcal{P}, L_{\mathcal{P},F}^{\mu}(H))|$ as the order of the coset group for a chosen partition $\mathcal P$ and real $\mu\in (0,1)$. See the following example.}

\begin{example}\label{ex:infinite}
Fix $t\in \mathbb{N}$ and let $F=K_{2,4}$ and $H=K_{n/2-t, n/2+t}$ with a vertex partition $\mathcal P$ of $H$ the natural bipartition of its vertex set. Note that $I_{\mathcal P, F}^{\mu}(H)=\{(2,4),(4,2)\}$ and $L_{\mathcal P, F}^{\mu}(H)$ consists of all pairs $(a,b)\in (2\mathbb{Z})^2$ with $a+b\in 6\mathbb{Z}$.
Moreover, $L_{\max}^2$ consists of all pairs $(a,b)\in \mathbb{Z}^2$ with $a+b\in 6\mathbb{Z}$, and it is easy to see that $|Q(\mathcal{P}, L_{\mathcal{P},F}^{\mu}(H))|=2$.
Changing the value of $t\in [-n/6, n/6]$ gives infinitely many choices of $H$ for $q=2$ and $\delta(H)\ge n/3=(1-1/\chi_{cr}(F))n$.

\end{example}

Let $F$ be an $r$-vertex $k$-graph and let $H$ be an $n$-vertex $k$-graph. 
We say that two vertices $u, v \in V(H)$ are \emph{$(F, \beta, i)$-reachable in $H$} if there exist at least $\beta n^{i r-1}$ sets $S \subseteq V(H)$ of size $i r-1$ such that both $H[S \cup \{u\}]$ and $H[S \cup \{v\}]$ contain an $F$-factor. 
Such a set $S$ is called a \emph{reachable $(ir-1)$-set for $u$ and $v$}.

A vertex set $U \subseteq V(H)$ is \emph{$(F, \beta, i)$-closed in $H$} if every pair of vertices $u, v \in U$ is $(F, \beta, i)$-reachable in $H$. 
For any vertex $v \in V(H)$, define $\tilde{N}_{F,\beta, i}(v,H)$ to be the set of vertices in $V(H)$ that are $(F,\beta,i)$-reachable to $v$ in $H$.
Let $\beta, c > 0$ and $t \in \mathbb{N}$. 
A partition $\mathcal{P} = \{V_1, \dots, V_d\}$ of $V(H)$ is called \emph{$(F,\beta,t,c)$-good} if it satisfies the following properties:
\begin{itemize}
    \item $V_i$ is $(F,\beta,t)$-closed in $H$ for every $i \in [d]$;
    \item $|V_i| \geq cn$ for every $i \in [d]$.
\end{itemize}

In the case of perfect matchings (i.e. when $F$ is an edge), we write $(\beta,i)$-reachable, $(\beta,i)$-closed, $\tilde{N}_{\beta, i}(v,H)$ and $(\beta,t,c)$-good for $(F,\beta,i)$-reachable, $(F,\beta,i)$-closed, $\tilde{N}_{F,\beta, i}(v,H)$ and $(F,\beta,t,c)$-good respectively.

We use $\ll$ to denote a hierarchy between constants.
If we write that a statement holds whenever $0 < a \ll b, c \ll d$, it means that there exist non-decreasing functions $g_1, g_2 \colon (0,1] \to (0,1]$ and $f \colon (0,1]^2 \to (0,1]$ such that the statement holds for all $a, b, c, d$ satisfying $b \le g_1(d)$, $c \le g_2(d)$, and $a \le f(b, c)$.
We will not explicitly compute these functions to avoid cluttering the presentation of the proofs.
We denote $a\in[b-c,b+c]$ by $a=b\pm c$ throughout the paper.

\subsection{Preliminaries}We first need the following concentration inequalities.
\begin{lemma}\cite[Theorem 2]{MR144363}\label{lem:Hoeffding}
    Let $X_1,\ldots,X_n$ be mutually independent random variables where each $X_i$ has Bernoulli distribution.
Consider $S_n=\sum_{i=1}^n X_i$,
then for any $t>0$,
\[
\mathbb{P}\left[|S_n-\mathbb{E}[S_n]|\ge t\right]
\le
2\exp\left(
-{2t^2}/{n}
\right).
\]
\end{lemma}

\begin{lemma}\cite[Lemma 1.2]{mcdiarmid}\label{lem:McDiarmid's_inequality}
	Let $ X_1,\dots,X_n $ be independent random variables, with each $ X_i $ taking values in a finite set $\Lambda_i$. Let $f:\prod_{i=1}^n \Lambda_i\rightarrow\mathbb{R}$ be a function satisfying: for some $ L>0 $ if $\mathbf{x},\mathbf{y}\in \prod_{i=1}^n \Lambda_i$ differ by at most one coordinate then $|f(\mathbf{x})-f(\mathbf{y})|\le L$. Then, for every $ t>0 $ there holds 
    \[\mathbb{P}[|f(X_1,\dots,X_n)-\mathbb{E}[f(X_1,\dots,X_n)]|\ge t]\le 2\exp(-2t^2/(nL^2)).
    \]
\end{lemma}
\begin{lemma}\cite[Lemma 6.1]{liebenau2023asymptotic}
\label{lem:hmc}
Let $c>0$, and let $f$ be a function defined on the set of subsets of some set $U$ such that $|f(U_1)-f(U_2)|\le c$ whenever $|U_1|=|U_2|=m$ and $|U_1\cap U_2|=m-1$.
Let $A$ be a uniformly random $m$-subset of $U$.
Then, for any $\alpha>0$, we have
\[\mathbb{P}\left[
  \left|f(A)-\mathbb{E}[f(A)]\right|
  \ge \alpha c\sqrt{m}
\right]
\le 2\exp(-2\alpha^2).\]
\end{lemma}
A \emph{lattice} in the Euclidean space $\mathbb{R}^d$ is an additive subgroup of $\mathbb{R}^d$ which is discrete.
In this paper we only consider lattices which are subsets of $\mathbb{Z}^d$.
A lattice in $\mathbb{R}^d$ is said to have \emph{full rank} if it contains $d$ linearly independent vectors.
A set $A$ in $\mathbb{R}^d$ is \emph{convex} if we have $(1-\theta)x+\theta y\in A$ whenever $x,y\in A$ and $0\le\theta\le1$.
We call $A$ a \emph{convex body} if it is convex, open, non-empty and bounded.
Next, we recall a result of Gauss concerning the intersection of a large convex body with a lattice of full rank

\begin{lemma}\cite[Lemma 3.22]{tao2006additive}\label{lem:lattice_points}
Let \(\Gamma\subseteq \mathbb R^d\) be a full-rank lattice, and let
\(B\subseteq \mathbb R^d\) be a convex body. Then, for all sufficiently
large \(R>0\), we have
\[
 |(R\cdot B)\cap \Gamma|
 =
 \bigl(R^d\pm O_{\Gamma,B,d}(R^{d-1})\bigr)
 \frac{\operatorname{mes}(B)}{\operatorname{covol}(\Gamma)}.
\]
Here \(\operatorname{covol}(\Gamma)\) denotes the volume of a fundamental domain of \(\Gamma\).
Moreover, the same estimate holds with \(\Gamma\) replaced by any translate \(\mathbf v+\Gamma\) of \(\Gamma\) where $\mathbf v\in \mathbb R^d$, with the same covolume \(\operatorname{covol}(\Gamma)\).
\end{lemma}
The final assertion follows from the same proof as the lattice case, since translating the lattice only translates the corresponding fundamental-domain tiling and does not change the covolume.
The following proposition tells us that if the coset group $\mathbb Z^{d}/A$  is finite, then 
$A$ must be a subgroup of $\mathbb Z^{d}$ of full rank $d$.
\begin{proposition}\label{prop:fullrank}
Let \(A\subseteq \mathbb Z^{d}\) be an additive subgroup.
If $\left|\mathbb Z^{d}/A\right|=c<\infty$, then \(A\) is a full-rank lattice in \(\mathbb Z^{d}\).
\end{proposition}
\begin{proof}
Let \(\textbf{e}_1,\dots,\textbf{e}_d\) be the standard basis of \(\mathbb Z^{d}\).
Since \(\mathbb Z^{d}/A\) is a finite abelian group of order \(c\),
every element of \(\mathbb Z^{d}/A\) has order dividing \(c\) by Lagrange's theorem.
In particular, for each \(i\in[d]\), $
c(\textbf{e}_i+A)=A$,
and hence $c\textbf{e}_i\in A$.
Thus \(A\) contains \(d\) linearly independent vectors $c\textbf{e}_1,\dots,c\textbf{e}_d$, and we are done.
\end{proof}

\section{A structural theorem}
\label{sec:generalfra}
Let $k,\ell \in \mathbb{N}$ where $\ell \le k-1$.
A $k$-graph $H=(V,E)$ consists of a vertex set of order $n$ and an edge set $E\subseteq\binom{V}{k}$.
For any $\ell$-subset $S$ of $V$ where $\ell\in[k-1]$, the \emph{degree} of $S$, denoted by $\deg_H(S)$ is the number of edges containing $S$.
The \emph{minimum} $\ell$-\emph{degree} $\delta_{\ell}(H)$ is the minimum of $\deg_H(S)$ over all $\ell$-subsets $S$ of $V$.
Let $F$ be an $r$-vertex $k$-graph and $D \in \mathbb{N}$. Define $\delta(F,\ell,D)$ to be the smallest number $\delta$ such that every $k$-graph $H$ on $n$ vertices with
$\delta_\ell(H) \ge \bigl(\delta + o_n(1)\bigr)\binom{n-\ell}{k-\ell}$
contains an $F$-packing covering all but at most $D$ vertices. 
We write $\delta(k,\ell,D)$ for $\delta(F,\ell,D)$ when $F$ is a single edge.
More generally, Han and Treglown \cite{han2020complexity} established the following general lattice-based criterion for factors in dense hypergraphs. 

\begin{theorem}\cite[Theorem 3.1]{han2020complexity}\label{thm:Han_structural_theorem}
         Let $k,\ell\in \mathbb{N}$ where $\ell\le k-1$ and let $F$ be an $r$-vertex $k$-graph. 
         Define $D,q,t,n_0\in\mathbb{N}$ and $\beta,\mu,\gamma,c>0$ where
         \[
         1/n_0\ll\beta,\mu\ll\gamma,c,1/r,1/D,1/q,1/t.
         \]
         Let $H$ be a $k$-graph on $n\ge n_0$ vertices where $r$ divides $n$. 
         Suppose that 
         \begin{enumerate}[label=(\roman*),font=\normalfont]
             \item $\delta_{\ell}(H)\ge (\delta(F,\ell,D)+ \gamma)\binom{n-\ell}{k-\ell}$;\label{item:\romannumeral1}
             \item $\mathcal{P}=\{V_1, \dots, V_d\}$ is an $(F,\beta,t,c)$-good partition of $V(H)$;\label{item:\romannumeral2}
             \item $|Q(\mathcal{P},L_{\mathcal{P},F}^{\mu}(H))|\le q$.\label{item:\romannumeral3}
         \end{enumerate}
         Then $H$ contains an $F$-factor if and only if $(\mathcal{P},L^{\mu}_{\mathcal{P},F}(H))$ is $q$-soluble.
\end{theorem}
Theorem~\ref{thm:Han_structural_theorem} is deterministic. 
A natural question is whether this lattice-based framework is robust under random vertex sampling. 
In particular, while one may expect local density conditions to inherit in a random induced subhypergraph, it is far less clear whether the global divisibility
information encoded by the lattice \(L^\mu_{\mathcal P,F}(H)\) continues to control the existence of
an \(F\)-factor after sampling. This is the question that we address in this paper.

The results about induced subgraphs mostly concern special spanning structures in graphs, whereas here we focus on a general framework about $F$-factors in random induced subhypergraphs.
In this paper, we focus on studying the probability that $H[p]$ contains an $F$-factor where $H$ is a dense $k$-graph satisfying a suitable minimum $\ell$-degree condition, and that $H$ admits an $(F,\beta,t,c)$-good partition $\mathcal P$ such that the associated coset group
$
Q(\mathcal P,L^\mu_{\mathcal P,F}(H))
$
has size $q$. 
We estimate the probability that the random induced subhypergraph $H[p]$ contains an $F$-factor as follows, yielding the number of subsets of $H$ that induce $F$-factors.

The following theorem is our main structural result.

\begin{theorem}\label{thm:main_structural_theorem_ell_degree}
Let $k,\ell\in \mathbb{N}$ with $\ell\le k-1$, and let $F$ be an $r$-vertex $k$-graph. 
Define $C',D,q,d,t,n_0\in\mathbb{N}$ and $\varepsilon,\rho,\beta,\gamma,\mu_0,\eta,c>0$ such that
\[
1/n_0\ll1/C'\ll\beta,\mu_0\ll\varepsilon,\rho,\gamma,c,\eta,1/r,1/D,1/q,1/d,1/t\]
together with \(\rho \ll\gamma,\eta\).
Let $H$ be a $k$-graph on $n\ge n_0$ vertices with \(\delta_{\ell}(H)\ge (\delta(F,\ell,D)+ \gamma)\binom{n-\ell}{k-\ell},\) and let $\mathcal{P}=\{V_0,V_1,\dots,V_d\}$ be a partition of \(V(H)\).
Suppose that
\begin{enumerate}[label=(\roman*),font=\normalfont]
    \item\label{item:V_0}
    $|V_0|\le \rho  n$, and every vertex in \(V_0\) lies in at least \(\eta n^{r-1}\) copies of \(F\);
    \item\label{item:V_i}
    for every \(i\in [d]\), $|V_i| \ge cn$, and $V_i$ is $(F,\beta,t)$-closed in \(H\left[\bigcup_{i\in[d]} V_i\right]\);

    \item\label{item:finite_coset}
    $|Q(\mathcal{P},L_{\mathcal{P},F}^{\mu_0}(H))|=q$.
\end{enumerate}
Let \(h:=\max\{2(tr-1),2r\}\). Suppose that \(p\ge C'(\log{n}/n)^{1/h}\) and \(p(1-p)\ge C'n^{-1/2}\). 
Then \(\mathbb{P}[H[p]\text{ contains an $F$-factor }]\ge \frac{1}{rq}-\varepsilon\).
In particular, if the index vector of every copy of \(F\) in \(H\) is $\mu_0$-robust, then \(\mathbb{P}[H[p]\text{ contains an $F$-factor }]= \frac{1}{rq} \pm \varepsilon\).
\end{theorem}

The quantity $1/rq$ in the probability has a natural arithmetic interpretation.
The factor \(1/r\)
corresponds to the necessary divisibility condition on the total number of sampled vertices, while
the factor \(1/q\) reflects the number of cosets of the lattice \(L^\mu_{\mathcal P,F}(H)\) inside
\(L^d_{\max}\). 
Roughly speaking, our argument shows that with probability asymptotic to \(1/(rq)\), the
sampled index vector lands in a residue class for which the induced lattice system is already soluble.

Our result may be viewed as a random induced analogue of the lattice-based theory for factors in dense hypergraphs. More broadly, it shows that the divisibility structure encoded by the lattice method is robust under random vertex deletion. It would be interesting to understand whether similar ideas can be used to obtain sharper probability estimates, threshold phenomena, or extensions to other random substructure models.
Next, we present the applications of Theorem \ref{thm:main_structural_theorem_ell_degree}.

\subsection{$F$-factors: a proof of Theorem~\ref{thm:unbalancedfactor}}
Shokoufandeh and Zhao~\cite{shokoufandeh2003proof} showed that $\delta(F,1,\allowbreak 5r^2)=1-1/\chi_{cr}(F)$. 
We also need that graphs $H$ with $\delta(H)\ge (1-1/{\chi}_{cr}(F)+\gamma)n$ admit a good partition $\mathcal P$ and the robust $F$-lattice defined on $\mathcal P$ has finite order, which were all indeed shown in \cite{han2020complexity}.
We summarize these results to the following lemma.

\begin{lemma}\cite{han2020complexity}\label{lem:factor}
    Let $k,r,n\ge 2$ be integers and $\mu,\gamma,\beta >0$ where $ 1/n\ll \beta,\mu\ll\gamma,1/k$.
    Let $F$ be an unbalanced $r$-vertex $k$-chromatic graph and $h:=r^{k-1}$. 
    For each $n$-vertex graph $H$ with $\delta(H)\ge (1-1/{\chi}_{cr}(F)+\gamma)n$, there exists an $(F,\beta,2^{h-1},1/r)$-good partition $\mathcal{P}$ of $ V(H)$. Moreover, \(|Q(\mathcal{P},L_{\mathcal{P},F}^{\mu}(H))|\le (2r-1)^r.\) 
\end{lemma}

Indeed, the first part of the conclusion was shown in the proof of \cite[Theorem 1.11]{han2020complexity}, and the second part is exactly \cite[Proposition 9.3]{han2020complexity}. 

Applying Lemma~\ref{lem:factor} together with Theorem~\ref{thm:main_structural_theorem_ell_degree}, we obtain Theorem \ref{thm:unbalancedfactor}. 

\begin{proof}[Proof of Theorem \ref{thm:unbalancedfactor}]
{If $F$ is balanced, then $\chi_{cr}(F)=\chi(F)$, and thus the result follows from Proposition~\ref{prop:simple} with $q=1$.}
It remains to consider an unbalanced $r$-vertex $k$-chromatic $F$.
     Suppose \(1/n_0\ll 1/C' \ll\beta,\mu\ll\varepsilon,\gamma,1/k\).
     Let \(H\) be a graph on \(n\ge n_0\) vertices with $\delta(H)\ge (1-1/{\chi}_{cr}(F)+\gamma)n$. 
     Applying Lemma~\ref{lem:factor}, we get an $(F,\beta,2^{r^{k-1}-1},1/r)$-good partition $\mathcal{P}$ of $V(H)$ such that $|Q(\mathcal{P},L_{\mathcal{P},F}^{\mu}(H))|\le (2r-1)^r$. 
     Let $q=|Q(\mathcal{P},L_{\mathcal{P},F}^{\mu}(H))|$.
     Together with the fact that $\delta(F,1,5r^2)=1-1/\chi_{cr}(F)$, this verifies assumptions of Theorem~\ref{thm:main_structural_theorem_ell_degree}. 
     Hence, applying Theorem~\ref{thm:main_structural_theorem_ell_degree} with \(V_0=\emptyset\), \(c=1/r\), and \(t=2^{r^{k-1}-1}\), we conclude that \[\mathbb{P}\left[H[p]\text{ contains an $F$-factor}\right]\ge\frac{1}{rq}-\varepsilon
     .\qedhere\]
\end{proof}

\subsection{Perfect matchings: a proof of Theorem~\ref{thm:perfectmatching}}

Han \cite{Han2015near} showed that $\delta(k,k-1,k)=1/k$. 
Therefore, to prove Theorem~\ref{thm:perfectmatching}, it suffices to obtain a suitable $(\beta,t,c)$-good partition together with an appropriate bound on the size of the associated coset group. 

The following lemma follows from \cite[Proposition~3.7]{han2017decision} and \cite[Lemma~3.8]{han2017decision}, with the error term \(-\gamma n\) there replaced by \(+\gamma n\).

\begin{lemma}\cite{han2017decision}\label{lem:perfectmatching}
    Suppose \(1/n\ll\beta,\mu\ll\gamma,1/k\), where \(k\ge 3\) is an integer. Let \(H\) be a \(k\)-graph on \(n\) vertices with $\delta_{k-1}(H)\ge (1/s+\gamma)n$ for integer \(s=2,\dots, k\). Then there exists a $(\beta,2^{s-2},1/s)$-good partition $\mathcal{P}$ of $V(H)$ such that \(|\mathcal{P}|\le s-1\). 
\end{lemma}

We also use the following result from~\cite{han2026perfect}, whose proof heavily relies on its predecessor from~\cite{KKM2015matchingpoly}. 

\begin{proposition}\cite[Lemma 7.1]{han2026perfect}
\label{prop:coset_group}
    Fix an integer \(k\ge 3\). Suppose \(\mu \ll c,1/k\). Let \(H\) be a \(k\)-graph on \(n\) vertices such that \(\delta_{k-1}(H)\ge n/k\), and let \(\mathcal{P}\) be a partition of \(V(H)\) in which each part has size at least \(cn\). Then \(|Q(\mathcal{P},L_{\mathcal{P}}^{\mu}(H))|\le |\mathcal{P}|\).
\end{proposition}

Now we are ready to prove Theorem \ref{thm:perfectmatching} by applying Theorem~\ref{thm:main_structural_theorem_ell_degree} and the tools above.

\begin{proof}[Proof of Theorem \ref{thm:perfectmatching}]
    Fix an integer \(s\in \{2,\dots,k\}\). Suppose that \(1/n_0\ll 1/C' \ll\beta,\mu\ll\varepsilon,\gamma,1/k.\)
    Let \(H\) be a \(k\)-graph on \(n\ge n_0\) vertices with $\delta_{k-1}(H)\ge (1/s+\gamma)n$. 
    We apply Lemma~\ref{lem:perfectmatching} to \(H\) to obtain a \((\beta, 2^{s-2},1/s )\)-good partition $\mathcal{P}$ of $V(H)$ with $|\mathcal{P}|\leq s-1$. 
    Then, by Proposition~\ref{prop:coset_group}, we have \(q=|Q(\mathcal{P},L_{\mathcal{P}}^{\mu}(H))|\le |\mathcal{P}|\).
    Together with the fact that $\delta(k,k-1,k)=1/k\le 1/s$, this verifies assumptions of Theorem~\ref{thm:main_structural_theorem_ell_degree}. Hence, applying Theorem~\ref{thm:main_structural_theorem_ell_degree} with \(V_0=\emptyset\), \(c=1/s\), and \(t=2^{s-2}\), we conclude that 
    \[
    \mathbb{P}\left[H[p]\text{ contains a perfect matching}\right]\geq \frac{1}{kq}-\varepsilon \geq \frac{1}{k(s-1)}-\varepsilon. \qedhere
    \]
\end{proof}

\subsection{Perfect matchings under minimum $\ell$-degree: a proof of Theorem~\ref{thm:ell_degree}}
We conclude this section by proving Theorem~\ref{thm:ell_degree}. 
We first recall two lemmas that will be used to construct a partition into a small exceptional set and a bounded number of closed parts.

\begin{lemma}\cite[Lemma 5.4]{gan2025keevash}\label{lem:s+1}
    Let \(\alpha>0\), and integers \(s,k\ge 2\) be given and suppose \(1/n\ll\delta'\ll\alpha, 1/k,1/s \). 
    Assume that \(H\) is a \(k\)-graph on \(n\) vertices satisfying that every set of \(s+1\) vertices contains two vertices that are \( (2\alpha,1)\)-reachable in \(H\). Then in time \(O(sn^{k+1})\) we can find a set of vertices \(S\subseteq V(H)\) with \(|S|\geq (1-s\delta')n\) such that \(|\tilde{N}_{\alpha,1}(v,H[S])|\geq \delta'n\) for any \(v\in S\).
\end{lemma}

\begin{lemma}\cite[Lemma 6.3]{han2020complexity}\label{lem:good_partition}
    Let \(\delta'>0\), and integers \(k, s\ge 2\) be given and suppose \( 1/n\ll\beta\ll\alpha\ll 1/s,\delta' \). 
    Assume \(H\) is an \(n\)-vertex \(k\)-graph and \(S\subseteq V(H)\) is such that \(|\tilde{N}_{\alpha,1}(v,H)\cap S| \geq\delta'n\) for any \(v\in S\). 
    Further, suppose every set of \(s+1\) vertices in \(S\) contains two vertices that are \((\alpha,1)\)-reachable in \(H\). 
    Then there exists a partition \(\mathcal{P}\) of \(S\) into \(V_1,\ldots,V_r\) with \(r\le\min\{s, 1/\delta'\} \) such that for any \(i\in [r] \), \( |V_i|\geq (\delta'-\alpha)n \) and \(V_i\) is \((\beta, 2^{s-1})\)-closed in \(H\).
\end{lemma}

Given such a partition, the following proposition bounds the size of the associated coset group. 
It follows from the proof of \cite[Proposition~4.1]{gan2025keevash}, after merging the parts \(V_0,\ldots,V_s\) appearing there into a single exceptional part \(V_0\) and omitting the second robust lattice.

\begin{proposition}\cite{gan2025keevash}\label{lemm:bound_ell}
    Suppose \(1/n\ll\mu \ll\beta\ll \delta'\ll\gamma\ll 1/k\). 
    Let \(H\) be an \(n\)-vertex \(k\)-graph with \(\delta_{\ell}(H)\ge (c^*_{k,\ell}+\gamma)\binom{n-\ell}{k-\ell}\). 
    Let \(s:=\lfloor 1/c_{k,\ell}^* \rfloor\). 
    Suppose that \(\mathcal{P}=\{V_0,V_1,\dots,V_d\}\) is a partition of \(V(H)\) satisfying \(|V_0|\le s\delta' n\) and for every \(i\in [d]\), $|V_i| \ge \delta'n/2$ and  $V_i$ is $(\beta,2^{s-1})$-closed in \(H\left[\bigcup_{i\in[d]} V_i\right]\).
    Then $|Q(\mathcal{P},L_{\mathcal{P}}^{\mu}(H))|\le (2k+1)^{d}$.
\end{proposition}

Finally, we recall the following almost-perfect matching bound.

\begin{lemma}\cite[Theorem 1.4]{ChangGeHanWang2022}
\label{lem:almostperfectfrac}
Given integers $1\leq \ell\leq k-1$, we have  \(\delta(k,\ell,2k-\ell-1)\le c_{k,\ell}^*\).
\end{lemma}

Now we are ready to prove Theorem \ref{thm:ell_degree}.

\begin{proof}[Proof of Theorem \ref{thm:ell_degree}]
Fix \(1\le \ell\le k-1\). Suppose that \[1/n_0\ll1/C'\ll\mu\ll\beta\ll\alpha\ll\delta'\ll\alpha'\ll \varepsilon,c_{k,\ell}^*,\gamma,1/k.\]
Let \(H\) be a \(k\)-graph on \(n\ge n_0\) vertices with $\delta_{\ell}(H)\ge (c_{k,\ell}^*+\gamma)\binom{n-\ell}{k-\ell}$.
Set \(s:=\lfloor 1/c_{k,\ell}^* \rfloor\).
Note that $c_{k,\ell}^* \ge c_{k,k-1}^*= 1/k$.
Then $s\le k$.
We have \[(s+1)\delta_1(H)\ge (s+1)(c_{k,\ell}^*+\gamma)\binom{n-1}{k-1}> (1+\gamma)\binom{n-1}{k-1},\]
which implies that every set of \(s+1\) vertices of \(V(H)\) contains two vertices that are \((2\alpha',1)\)-reachable. 
Indeed, otherwise by the inclusion-exclusion principle and \(\alpha'\ll \gamma\), we have 
\[\binom{n}{k-1}\ge (s+1)\delta_1(H)-\binom{s+1}{2}\cdot 2\alpha'n^{k-1}>(1+\gamma)\binom{n-1}{k-1}-(s+1)^2\alpha' n^{k-1}\ge \binom{n}{k-1}\]
a contradiction. 

If $s=1$, then every two vertices of $V(H)$ are $(2\alpha',1)$-reachable.
Thus $V(H)$ is $(\beta,1)$-closed and $V_0=\emptyset$.
Consequently, $\mathcal P:=\{\emptyset, V(H)\}$ satisfies properties~\ref{item:V_0}-\ref{item:V_i} of Theorem~\ref{thm:main_structural_theorem_ell_degree} with \(t=1\).
For $s\ge2$, by Lemma~\ref{lem:s+1}, we find \(U\subseteq V(H)\) with \(|U|\ge (1-s\delta')n\) such that \(|\tilde{N}_{\alpha',1}(v,H[U])|\ge \delta'n\) for every \(v\in U\). 
Let \(V_0:=V(H)\setminus U\) and thus \(|V_0|\le s\delta' n\). 
Since \(\alpha<\alpha'\), for every \(v\in U\), we have \(|\tilde{N}_{\alpha,1}(v,H[U])|\ge |\tilde{N}_{\alpha',1}(v,H[U])|\ge \delta'n\).
Moreover, every set of \(s+1\) vertices of \(U\) contains two vertices that are \((2\alpha,1)\)-reachable in $H[U]$ since $|V_0|\le s\delta'n$ and $\alpha\ll\delta'\ll\alpha'$. 
Apply Lemma~\ref{lem:good_partition} to \(H[U]\), we find a partition \(\mathcal{P}_1\) of \(U\) into \(V_1,\dots,V_d\) with \(d\le s\) such that for \(i\in [d]\), \(|V_i|\ge (\delta '-\alpha)|U|\ge \delta' n/2\) and \(V_i\) is \((\beta, 2^{s-1})\)-closed in \(H[U]\). 
Moreover, every vertex \(v\in V_0\) lies in at least
\(\delta_1(H)\ge (c_{k,\ell}^*+\gamma)\binom{n-1}{k-1}\) edges. 
Consequently, \(\mathcal P:=\{V_0,V_1,\dots,V_d\}\)
satisfies properties~\ref{item:V_0}-\ref{item:V_i} of Theorem~\ref{thm:main_structural_theorem_ell_degree} with \(t=2^{s-1}\). 
By Proposition~\ref{lemm:bound_ell} and the fact that \(d\le s\le k\), we have \(q=|Q(\mathcal P,L_{\mathcal P}^{\mu}(H))|\le(2k+1)^d \le (2k+1)^k\). 
As \(\delta(k,\ell,2k-\ell-1)\le c_{k,\ell}^*\) by Lemma \ref{lem:almostperfectfrac}, 
this verifies all assumptions of Theorem~\ref{thm:main_structural_theorem_ell_degree}. Applying that theorem with \(t=2^{s-1}\), we conclude that 
\[\mathbb{P}\left[H[p]\text{ contains a perfect matching}\right]\geq \frac{1}{kq}-\varepsilon \geq \frac{1}{k(2k+1)^k}-\varepsilon.\qedhere\]
\end{proof}

\section{Proof of Theorem \ref{thm:main_structural_theorem_ell_degree}}
\label{sec:proofofmain}

\subsection{Proof sketch of Theorem~\ref{thm:main_structural_theorem_ell_degree}}
The proof of Theorem \ref{thm:main_structural_theorem_ell_degree} has two main ingredients. 
First, we establish an inheritance statement showing that, with high probability, the random induced subhypergraph $H[p]$ ``inherits'' the structural hypotheses from $H$: the minimum $\ell$-degree condition, the good partition structure, and the relevant lattice information. 
Second, we estimate the probability that the index vector of the sampled vertex set lies in the original lattice $L^\mu_{\mathcal P,F}(H)$. This step combines concentration inequalities with a lattice-point counting argument for large boxes intersected with a full-rank lattice. 
Together, these two ingredients imply that with probability asymptotic to $1/(rq)$ the induced lattice system is already soluble via the empty tiling, and Theorem \ref{thm:Han_structural_theorem} then yields an $F$-factor in $H[p]$.


\begin{lemma}\label{lem:inheritance_random_subset}
		Let $k,\ell\in \mathbb{N}$ where $\ell\le k-1$ and let $F$ be an $r$-vertex $k$-graph. Define $C,t,d,n_0\in\mathbb{N}$ and $\mu_0,\eta,\rho ,\beta,\gamma,c\in (0,1)$ where
         \[
         1/n_0\ll 1/C\ll\eta,\rho ,\beta,\gamma,\mu_0,c,1/k,1/d, 1/r,1/t.
         \] 
        together with $1/n_0\ll\mu_0\ll1/k,1/r$.
        For any constant $\delta$, let $H$ be a $ k $-graph on $ n\ge n_0 $ vertices such that $\delta_{\ell}(H)\ge (\delta+ \gamma)\binom{n-\ell}{k-\ell}$, and let $\mathcal{P}=\{V_0,V_1, \dots, V_d\}$ be a partition of $V(H)$ satisfying the following properties:
         \begin{enumerate}
             \item $|V_0|\le \rho  n$, and for every \(v\in V_0\), there exists an \((r-1)\)-vector \(\mathbf{w}_v\) such that at least \(\eta n^{r-1}\) \((r-1)\)-sets \(T\) satisfy \(\mathbf{i}_{\mathcal P}(T)=\mathbf{w}_v\) and \(T\cup\{v\}\) spans a copy of \(F\) in \(H\);
             \item for every \(i\in [d]\), $|V_i| \ge cn$, and $V_i$ is $(F,\beta,t)$-closed in \(H\left[\bigcup_{i\in[d]} V_i\right]\).
         \end{enumerate}
         Suppose that \(p\ge C(\log{n}/n)^{1/h}\), where \(h:=\max\{2(tr-1),2r\}\).
         Let \(H':=H[p]\), \(V_i':=V_i\cap V(H')\) for every \(i\in\{0,\dots,d\}\), and let \(\mathcal{P}':=\{V_0',V_1',\dots,V_d'\}\). 
         Set \(m:=|V(H')|\).
Then with probability \(1-o_n(1)\), \(H'\) satisfies:
		\begin{enumerate}[label=(V\arabic*),font=\normalfont]
			\item $\delta_{\ell}(H')\ge (\delta+ \gamma/2)\binom{m-\ell}{k-\ell}$;\label{item:degree_inheritance}
            \item $|V_0'|\le 2\rho  m$ and for every \(v\in V_0'\), there are at least \(\eta m^{r-1}/2\) \((r-1)\)-sets \(T\) satisfy \(\mathbf{i}_{\mathcal P}(T)=\mathbf{w}_v\) and \(T\cup\{v\}\) spans a copy of \(F\) in \(H'\);\label{item:copy_inheritance}
            \item for every \(i\in [d]\), $|V_i'| \ge cm/2$, and $V_i'$ is $(F,\beta/2,t)$-closed in \(H'\Bigl[\bigcup_{i\in[d]} V_i'\Bigr]\);\label{item:reachable_inheritance}
            \item 
            there exists a constant \(\mu^*\) such that \(4^{-\binom{r+d-1}{r}}\mu_0\le \mu^* \le \mu_0\) and \(I_{\mathcal{P},F}^{\mu^*}(H)= I_{\mathcal{P}',F}^{\mu^*/2}(H').\) 
           \label{item:vector_inheritance}
		\end{enumerate}	
\end{lemma}

The second lemma estimates the solubility probability on the random induced subgraph.
\begin{lemma}\label{lem:lattice_prob}
   Let $k\in \mathbb{N}$, and let $F$ be an $r$-vertex $k$-graph. 
         Define $C,q,d,n_0\in\mathbb{N}$ and $\varepsilon,\mu_0,c>0$ where
         \[
         1/n_0\ll 1/C\ll\varepsilon,c,1/r,1/q,1/d\ \text{ and}\ 1/n_0\ll\mu_0\ll1/r.
         \]
         Let $H$ be a $k$-graph on vertex set $V$ with $|V|=n\ge n_0$. 
         Suppose that $p(1-p)\ge 2C^2n^{-1/2}$ and
         \begin{enumerate}[label=(\roman*),font=\normalfont]
             \item $\mathcal{P}=\{V_0,V_1, \dots, V_d\}$ is a partition of $V$ where $|V_i| \geq cn$ for every $i \in [d]$;
             \item $|Q(\mathcal{P},L_{\mathcal{P},F}^{\mu_0}(H))|=q$.
         \end{enumerate}
        Then for every coset \(\mathbf{v}+L_{\mathcal{P},F}^{\mu_0}(H)\in \mathbb Z^d/L_{\mathcal{P},F}^{\mu_0}(H)\), we have 
        \(\mathbb{P}[\mathbf{i}_{\mathcal{P}}(V_p)\in \mathbf{v}+L_{\mathcal{P},F}^{\mu_0}(H)]= \frac{1}{rq}\pm\varepsilon\).
\end{lemma}

We are now ready to prove the main theorem.
\begin{proof}[Proof of Theorem \ref{thm:main_structural_theorem_ell_degree}]
Suppose that \(H\) and \(\mathcal{P}\) satisfy the assumptions of the theorem. 
Suppose that \(p\ge C'(\log{n}/n)^{1/h}\) with \(h:=\max\{2(tr-1),2r\}\) and \(p(1-p)\ge C'n^{-1/2}\), which allows us to apply Lemmas~\ref{lem:inheritance_random_subset} and~\ref{lem:lattice_prob}. 
    Let $V_p$ be a $p$-random subset of $V(H)$ and \(m:=|V_p|\).
    Let \(H':=H[V_p]\), \(V_i':=V_i\cap V_p\) for every \(i\in\{0,\dots,d\}\), and \(\mathcal{P}':=\{V_0',V_1',\dots,V_d'\}\). 
    For every \(v\in V_0\), by property~\ref{item:V_0} and the pigeonhole principle, there exists an \((r-1)\)-vector \(\mathbf w_v\in \mathbb N^d\) such that at least 
    \[\frac{\eta n^{r-1}-r!\cdot \rho  n\cdot n^{r-2}}{r!\binom{r+d-1}{r}}\ge\frac{\eta}{2(r+d)^r}n^{r-1}\] 
    \((r-1)\)-sets \(T\subseteq \bigcup_{i\in[d]}V_i\) satisfy \(\mathbf i_{\mathcal P}(T)=\mathbf w_v\) and \(T\cup\{v\}\) spans a copy of \(F\).
    Define the event $B$ that $\mathbf i_{\mathcal P}(V_p)\in\sum_{v\in V_0'}\mathbf w_v+L_{\mathcal{P},F}^{\mu_0}(H)$.
Note that
$\sum_{v\in V_0'}\mathbf w_v+L_{\mathcal{P},F}^{\mu_0}(H)$
is a random coset, as it depends on the random set $V_0'$.
We claim that $\mathbb P[B]=\frac{1}{rq}\pm\frac{\varepsilon}{2}$.

For every fixed subset $Z\subseteq V_0$, define the coset $C_Z:=
\sum_{v\in Z}\mathbf w_v+L_{\mathcal{P},F}^{\mu_0}(H)$.
Condition on the event $V_p\cap V_0=Z$ which fixes $C_Z$. 
Note that $\mathbf i_{\mathcal P}(V_p)=\bigl(
|V_p\cap V_1|,\ldots,|V_p\cap V_d|
\bigr)$ is independent with the event
$V_p\cap V_0=Z$. 
Hence, by Lemma~\ref{lem:lattice_prob} with $\varepsilon/2$ in place of $\varepsilon$,
we have $\mathbb P\left[\mathbf i_{\mathcal P}(V_p)\in C_Z\,\middle|\,V_p\cap V_0=Z\right]=\frac{1}{rq}\pm\frac{\varepsilon}{2}.
$
Since this estimate holds uniformly for every $Z\subseteq V_0$,
the law of total probability gives
\begin{align*}
\mathbb P[B]=
\sum_{Z\subseteq V_0}
\mathbb P[V_p\cap V_0=Z]\,
\mathbb P\left[
\mathbf i_{\mathcal P}(V_p)\in C_Z
\,\middle|\,
V_p\cap V_0=Z
\right]=
\frac{1}{rq}\pm\frac{\varepsilon}{2}.
\end{align*}
    
   Next, let \(A\) be the event that the following holds:
    \begin{enumerate}[label=(V\arabic*)]
			\item $\delta_{\ell}(H')\ge (\delta(F,\ell,D)+ \gamma/2)\binom{m-\ell}{k-\ell}$;
            \item $|V_0'|\le 2\rho  m$ and for every \(v\in V_0'\), there are at least \(\frac{\eta}{4(r+d)^r} m^{r-1}\) \((r-1)\)-sets \(T\) satisfying \(\mathbf{i}_{\mathcal P}(T)=\mathbf{w}_v\) and \(T\cup\{v\}\) spans a copy of \(F\) in \(H'\);
            \item for every \(i\in [d]\), $|V_i'| \ge cm/2$, and $V_i'$ is $(F,\beta/2,t)$-closed in \(H'\Bigl[\bigcup_{i\in[d]} V_i'\Bigr]\);
            \item there exists $\mu^*$ such that $4^{-\binom{r+d-1}{r}}\mu_0\le\mu^*\le\mu_0$ such that \(I_{\mathcal{P},F}^{\mu^*}(H)= I_{\mathcal{P}',F}^{\mu^*/2}(H').\) 
		\end{enumerate}	
        By Lemma~\ref{lem:inheritance_random_subset} with $\eta/(2(r+d)^r)$ in place of $\eta$, we have $\mathbb{P}[A]=1-o_n(1)$. 
        Therefore, we have $\mathbb P[A\wedge B] \ge 1/(rq)-\varepsilon$.
        
Now suppose that both $A$ and $B$ hold, and we shall show that $H'$ contains an $F$-factor.
        This will prove the theorem as $\mathbb P\bigl[H[p]\text{ contains an }F\text{-factor}\bigr] \ge \mathbb P[A\wedge B] \ge 1/(rq)-\varepsilon$. 

        Let \(U':=\bigcup_{i\in [d]}V_i'=V_p\setminus V_0'\). 
        We randomly partition \(U'\) into two parts of almost equal size, denoted by \(U^+\) and \(U^-\). 
        By applying Lemma \ref{lem:hmc}  to each of the following properties and taking a union bound, with probability at least \(1-\exp(-\Omega(m))\), each of the following holds:
\begin{enumerate}[label=(U\arabic*)]
    \item for every \(v\in V_0'\), there are at least \(\frac{\eta}{(4r+4d)^r} m^{r-1}\) \((r-1)\)-sets \(T\subseteq U^+\) satisfying \(\mathbf{i}_{\mathcal P}(T)=\mathbf{w}_v\) and \(T\cup\{v\}\) spans a copy of \(F\) in \(H'\);\label{item:U^+}
    \item for every \(i\in [d]\), we have \(|V_i'\cap U^-|\ge cm/10 \), and for every \(u,v\in V_i'\), there are at least \(\beta^2 m^{tr-1}\) reachable sets \(S\subseteq U^-\) for \(u\) and \(v\);\label{item:U^-_reachable}
    \item for every $\mathbf{v} \in I_{\mathcal{P}',F}^{\mu^*/2}(H')$, \(H'[U^-]\) contains at least \((\mu^*)^2m^{r}\) copies of \(F\)  with index vector \(\mathbf{v}\).\label{item:U^-_lattice}
\end{enumerate}
Fix a partition satisfying the above properties. 

\medskip
\textbf{Step 1: Covering all vertices in \(V_0'\).} 
First, we construct an \(F\)-packing \(M_1\) using only vertices in \(U^+\) to cover all vertices in \(V_0'\). 
Let \(v_1,\dots, v_{|V_0'|}\) be an enumeration of \(V_0'\), noting \(|V_0'|\le 2\rho  m\). 
We cover \(V_0'\) using a greedy algorithm: 
process these vertices in order, and for each vertex \(v_i\), choose a copy \(F_i\) of \(F\) containing \(v_i\) and \(r-1\) vertices in \(U^+\), such that \(\mathbf{i}_{\mathcal P}(V(F_i)\setminus\{v_i\})=\mathbf{w}_{v_i}\) and \(F_i\) is vertex-disjoint from all previously chosen \(F_j\) for \(j<i\). 
This procedure can be carried out. 
Indeed, each previously chosen copy \(F_j\) eliminates at most \((r-1)m^{r-2}\) possible choices for \(F_i\). 
By~\ref{item:U^+} and \(\rho \ll \eta\), at each step there remain at least \(\frac{\eta}{(4r+4d)^r} m^{r-1}-2\rho  m\cdot (r-1)m^{r-2}\ge \frac{\eta}{(5r+5d)^r} m^{r-1}\) choices available for \(F_i\). 
Moreover, by construction, $\mathbf{i}_{\mathcal P}(V(M_1))=\sum_{v\in V_0'}\mathbf w_v$.
Since $B$ holds, we have $\mathbf{i}_{\mathcal P}(V_p)\in\sum_{v\in V_0'}\mathbf w_v+L_{\mathcal P,F}^{\mu_0}(H)=\mathbf{i}_{\mathcal P}(V(M_1))+L_{\mathcal P,F}^{\mu_0}(H)$.
Consequently, $\mathbf{i}_{\mathcal P}(V_p\setminus V(M_1))\in L_{\mathcal P,F}^{\mu_0}(H)$.

\medskip
\textbf{Step 2: Finding an $F$-factor in the remaining hypergraph.} 
Since all vertices in \(V_0'\) are covered by \(M_1\), we restrict attention to the remaining parts. 
Next, consider the remaining hypergraph \(H'':=H'[V_p\setminus V(M_1)]\) and its partition \(\mathcal{P}'':=\{V_i'\cap V(H'')\}_{i\in [d] }\). 
Note that \(|V(H'')|\ge m-r\cdot 2\rho  m>m/2\). 
We verify that \(H''\) and \(\mathcal{P}''\) satisfy the assumptions of Theorem \ref{thm:Han_structural_theorem}. 

For the first assumption,
by (V1), we get \[\delta_{\ell}(H'')\ge \delta_{\ell}(H')-r\cdot2\rho  m\cdot m^{k-\ell-1}\ge (\delta(F,\ell,D)+ \gamma/4)\binom{|V(H'')|-\ell}{k-\ell},\] as desired.
The second assumption is guaranteed by~\ref{item:U^-_reachable} and the fact that \(U^- \subseteq V(H'')\), which together show that 
\(\mathcal{P}''\) is a \(\bigl(\beta^2,\, t,\, c/10\bigr)\)-good partition of \(V(H'')\).
We now verify the third assumption. 
Since every vector in
\(I_{\mathcal P',F}^{\mu^*/2}(H')\)
still occurs as the index vector of at least
\((\mu^*)^2m^r\) copies of \(F\) in \(H''\)
by \ref{item:U^-_lattice}, we have \(I_{\mathcal{P},F}^{\mu^*}(H)= I_{\mathcal{P}',F}^{\mu^*/2}(H')\subseteq I_{\mathcal{P}'',F}^{(\mu^*)^2}(H'')\). 
It follows that 
\[|Q(\mathcal{P}'',L_{\mathcal{P}'',F}^{(\mu^*)^2}(H''))|=|L_{\max}^{d}/L_{\mathcal{P}''}^{(\mu^*)^2}(H'')|\le |L_{\max}^{d}/L_{\mathcal{P},F}^{\mu^*}(H)|\le q.\] 
Finally, by the definition of \(H''\), we have $\mathbf{i}_{\mathcal P''}(V(H''))=\mathbf{i}_{\mathcal P}\bigl(V_p\setminus V(M_1)\bigr)\in L_{\mathcal P,F}^{\mu_0}(H)\subseteq L_{\mathcal P,F}^{\mu^*}(H)\subseteq L_{\mathcal P'',F}^{(\mu^*)^2}(H'')$ as $\mu^*\le\mu_0$.

Therefore,
\((\mathcal P'',L_{\mathcal P'',F}^{(\mu^*)^2}(H''))\)
is \(0\)-soluble. Moreover, since
\(L_{\mathcal P'',F}^{(\mu^*)^2}(H'')\subseteq L_{\max}^d\),
we have $r\mid\left|\mathbf{i}_{\mathcal P''}(V(H''))\right|=|V(H'')|$.
Hence all the assumptions of Theorem~\ref{thm:Han_structural_theorem} are satisfied, and \(H''\)
contains an \(F\)-factor \(M_2\), which yields that $M_1\cup M_2$ is an $F$-factor of $H'$.

We now show the last part of the theorem.
Suppose that $V_0=\emptyset$ and the index vector of every copy of $F$ in $H$ is $\mu_0$-robust. 
Let
$L:=L_{\mathcal P,F}^{\mu_0}(H)$.
If $H[p]$ contains an $F$-factor
$\mathcal M$, then we have $\mathbf i_{\mathcal P}(V_p)
=\sum_{F'\in\mathcal M}\mathbf i_{\mathcal P}(V(F'))\in L$.
Hence, 
by Lemma \ref{lem:lattice_prob} 
we have $\mathbb P[H[p]\text{ contains an $F$-factor}]
\leq \mathbb P[\mathbf i_{\mathcal P}(V_p)\in L]\le\frac1{rq}+{\varepsilon}$.
Together with $\mathbb P[H[p]\text{ contains an $F$-factor}]\ge\frac{1}{rq}-\varepsilon$, we obtain the desired result.
\end{proof}

\subsection{Proof of Lemma~\ref{lem:inheritance_random_subset}}
In this subsection, we prove Lemma \ref{lem:inheritance_random_subset} using concentration inequalities.

\begin{lemma}\label{lem:application_McDiarmid}
    Let $s,n\in \mathbb{N}$ and \(1/n\ll\xi\ll1/s\).
    Let $V$ be an $n$-vertex set and $\mathcal S\subseteq \binom{V}{s}$ satisfying $|\mathcal S|= \alpha n^s$. 
    Let $V_p\subseteq V$ be obtained by including each vertex independently with probability $p$, and define \(\mathcal S' := \{S\in \mathcal S : S\subseteq V_p\}.\) 
    Then $\mathbb{P}[|\mathcal S'|\in [(\alpha-2s\xi) |V_p|^s, (\alpha+7s\xi) |V_p|^s]]
\ge 1- 2\exp{(-\xi^2p^{2s}n)}$.
\end{lemma}

\begin{proof}
 Let \(V=\{v_1,\dots,v_n\}\), and let \(X_i\) be the indicator variable of the event \(v_i\in V_p\). Clearly, \(X_1,\dots,X_n\) are independent random variables. 
 Let $E_0$ be the event that $|V_p|= (1\pm \xi)pn$.
Note that by Lemma~\ref{lem:Hoeffding}, we have $\mathbb{P}[E_0]\ge 1-2\exp{\left(-2\xi^2p^2n\right)}$.

 Now let \(X=(X_1,\dots,X_n)\) and \(f(X):=|\mathcal S'|\). 
    If \(\mathbf{x},\mathbf{y}\in \{0,1\}^n\) differ in at most one coordinate, then \(|f(\mathbf{x})-f(\mathbf{y})| \le n^{s-1}.\) 
    By linearity of expectation 
    we have \(\mathbb{E}[f(X)] = \alpha n^s p^s\). 
    Let \(E_1\) be the event that $|\mathcal S'|=(\alpha\pm s\xi)(pn)^s$.
    By Lemma \ref{lem:McDiarmid's_inequality}, we get $\mathbb{P}\left[E_1^c\right]\le 2\exp{(-2(s\xi)^2p^{2s}n)}$.

Suppose both $E_0$ and $E_1$ hold, we have $|\mathcal S'|
\ge
\frac{\alpha-s\xi}{(1+\xi)^s}|V_p|^s
\ge
(\alpha-2s\xi)|V_p|^s$
and
$|\mathcal S'|
\le
\frac{\alpha+s\xi}{(1-\xi)^s}|V_p|^s
\le
(\alpha+7s\xi)|V_p|^s$,
where the last inequalities follow from
\(\alpha\le 1\) and \(\xi\ll 1/s\).
Consequently,
$\mathbb P\Big[|\mathcal S'|\notin\big[(\alpha-2s\xi)|V_p|^s, (\alpha+7s\xi)|V_p|^s \big]\Big]\le
\mathbb P[E_0^c]+\mathbb P[E_1^c]\le4\exp(-2\xi^2p^{2s}n)$.
Hence, 
\[\mathbb P\left[|\mathcal S'|\in\big[
(\alpha-2s\xi)|V_p|^s,(\alpha+7s\xi)|V_p|^s\big]\right]\ge1-2\exp(-\xi^2p^{2s}n).\qedhere \]
\end{proof}

We are now ready to prove Lemma~\ref{lem:inheritance_random_subset}.

\begin{proof}[Proof of Lemma~\ref{lem:inheritance_random_subset}]

Define an additional constant \(\xi>0\) such that \[1/n_0\ll 1/C\ll  \xi\ll\eta,\rho ,\beta,\mu_0,\gamma,c,1/k,1/d, 1/r,1/t.\] 
Suppose that \(p\ge C(\log{n}/n)^{1/h}\), where \(h:=\max\{2(tr-1),2r\}\).
Let $H':=H[p]$ and $m:=|V(H')|$.
We will estimate the probability that each condition fails using Lemma~\ref{lem:application_McDiarmid}, each being \(o_n(1)\).

For~\ref{item:degree_inheritance}, for any \(\ell\)-set \(S\subseteq V(H)\), let \(\mathcal S\) be the family of $(k-\ell)$-sets $T\subseteq V(H)$ such that $T\cup S \in E(H)$, and thus 
\(|\mathcal S|\ge \delta_\ell(H)\ge (\delta+\gamma)\binom{n-\ell}{k-\ell}\ge (\delta+3\gamma/4)\frac{n^{k-\ell}}{(k-\ell)!}\). 
Set  \(\mathcal S' := \{T\in \mathcal S : T\subseteq V(H')\}.\) By Lemma~\ref{lem:application_McDiarmid}, we have 
\[
P\left[|\mathcal S'|\le(\delta+ \gamma/2)\binom{m-\ell}{k-\ell}\right]\le P\left[|\mathcal S'|\le(\delta+ \gamma/2)\frac{{m}^{k-\ell}}{(k-\ell)!}\right] \le \exp{(- \xi^2p^{2(k-\ell)}n)}.
\]
By a union bound over all \(\ell\)-sets in \(V(H)\),  the probability that \ref{item:degree_inheritance} fails is at most \(\binom{n}{\ell}\exp{(- \xi^2p^{2(k-\ell)}n)}= o_n(1)\). 

We first record estimates for the sizes of the induced parts \(V_0',V_1',\dots,V_d'\). 
By Lemma~\ref{lem:application_McDiarmid}, applied with \(s=1\) and \(\mathcal{S}=V_j\) for $j\in \{0,1,\dots,d\}$, we obtain 
\[
\mathbb{P}\left[|V_0\cap V(H')|\ge 2\rho  m\right]\le \exp(-\xi^2p^2n)\le \exp{(-\xi^2C\log n)} =o_n(1)
\]
and similarly $\mathbb{P}\left[|V_j\cap V(H')|\le c m/2\right]\le \exp(-\xi^2p^2n)=o_n(1)$.

Next, we verify~\ref{item:copy_inheritance}. 
The first assertion in~\ref{item:copy_inheritance} follows as above.
For the second assertion in~\ref{item:copy_inheritance}, for every \(v\in V_0\), let \(\mathcal S_{v}\) be the family of \((r-1)\)-sets \(T\) such that \(\mathbf{i}_{\mathcal P}(T)=\mathbf w_v\) and \(T\cup\{v\}\) spans a copy of \(F\) in \(H\). 
By assumption, \(|\mathcal S_v|\ge \eta n^{r-1}\). 
Set  \(\mathcal S_{v}' := \{S\in \mathcal S_{v} : S\subseteq V(H')\}.\) 
Applying Lemma~\ref{lem:application_McDiarmid}, we obtain  
\[\mathbb{P}\left[|\mathcal S_{v}'|\le\eta m^{r-1}/2\right]\le \exp{(- \xi^2p^{2(r-1)}n)} \le \exp{(-\xi^2C\log n)} \le n^{-2}.\]
Taking a union bound over all vertices in \(V_0\), we obtain the desired conclusion.

Now, we verify \ref{item:reachable_inheritance}. 
Recall that $\mathbb{P}\left[|V_j\cap V(H')|\le c m/2\right]=o_n(1)$.
Let $W=\bigcup_{i\in[d]} V_i$ and note that \(\left|W\right|\ge dcn\ge \xi n\). 
For every \(u,v\in V_j\), let \(\mathcal S_{u,v}\) be the family of reachable \((tr-1)\)-sets for \(u\) and \(v\) in \(H\left[W\right]\). 
Since \(V_j\) is \((F,\beta,t)\)-closed in \(H\left[W\right]\), \(|\mathcal S_{u,v}|\ge \beta |W|^{tr-1}\). 
Set  \(\mathcal S_{u,v}' := \{S\in \mathcal S_{u,v} : S\subseteq W\cap V(H')\}\) and \(m_0:=\bigl|\bigcup_{i\in[d]} V_i'\bigr|\). 
Applying Lemma~\ref{lem:application_McDiarmid} with \(W\), we obtain 
\[
\mathbb{P}\left[|\mathcal S_{u,v}'|\le\beta m_0^{tr-1}/2\right]\le\exp{(- \xi^2p^{2(tr-1)}|W|)} \le \exp{(- \xi^3p^{2(tr-1)}n)} \le \exp{(-\xi^3C\log n)}\le n^{-3}.
\]
A union bound over all \(j\in [d]\) and all pairs \(u,v\in V_j\) yields the desired bound. 

Finally, we consider~\ref{item:vector_inheritance}. 
We need the following claim.
\begin{claim}
\label{clm:equalrobust}
    There exists a constant \(\mu^*\) such that \(4^{-\binom{r+d-1}{r}}\mu_0\le \mu^* \le \mu_0\) and \(I_{\mathcal{P},F}^{\mu^*}(H)=I_{\mathcal{P},F}^{\mu^*/4}(H)\).
\end{claim}
\begin{proof}
Note that \(|I_{\mathcal{P},F}^{\mu^*}(H)|\le \binom{r+d-1}{r}\). 
Starting with \(\mu^*:=\mu_0\), we iteratively replace \(\mu^*\) by \(\mu^*/4\) as long as \(I_{\mathcal{P},F}^{\mu^*}(H) \neq I_{\mathcal{P},F}^{\mu^*/4}(H).\) 
At each step, \(|I_{\mathcal{P},F}^{\mu^*}(H)|\) strictly increases, so at least one new \(r\)-vector is added. 
Since there are at most \(\binom{r+d-1}{r}\) such vectors, the process terminates after at most \(\binom{r+d-1}{r}\) steps. 
Let \(\mu^*\) be the resulting value. 
Then \(I_{\mathcal{P},F}^{\mu^*}(H)=I_{\mathcal{P},F}^{\mu^*/4}(H),\) 
and by construction \(\mu^* \ge 4^{-\binom{r+d-1}{r}} \mu_0.\) \qedhere
\end{proof}
Apply Claim \ref{clm:equalrobust} to $\mu_0$, we obtain that there exists a constant \(\mu^*\) such that \(4^{-\binom{r+d-1}{r}}\mu_0\le \mu^* \le \mu_0\) and \(I_{\mathcal{P},F}^{\mu^*}(H)=I_{\mathcal{P},F}^{\mu^*/4}(H)\).
Let \(E_1\) be the event that $I_{\mathcal{P}, F}^{\mu^*}(H) \subseteq I_{\mathcal{ P}', F}^{\mu^*/2}(H')$, and let \(E_2\) be the event that $I_{\mathcal{ P}', F}^{\mu^*/2}(H')\subseteq I_{\mathcal{ P}, F}^{\mu^*/4}(H)$.
Note that if both \(E_1\) and \(E_2\) hold, then we have $I_{\mathcal{ P}, F}^{\mu^*}(H) = I_{\mathcal{ P}', F}^{\mu^*/2}(H')$.
It therefore suffices to show that 
\(\mathbb{P}[E_1\wedge E_2] = 1- o_n(1)\).

Take any \(r\)-vector $\mathbf{v}$ and let \(\mathcal S_{\textbf{v}}\) be the family of copies of $F$ in \(H\) with index vector \(\textbf{v} \). 
Set \(|\mathcal S_{\textbf{v}}|=\alpha n^r\) and \(\mathcal S_{\textbf{v}}' := \{F\in\mathcal{S}_{\textbf{v}} : V(F)\subseteq V(H')\}.\) 
Let $Y_{\mathbf v}$ be the event that $|\mathcal S_{\textbf{v}}'|=(\alpha \pm \mu^*/4)m^r$.

By Lemma~\ref{lem:application_McDiarmid}, we obtain 
\[
\mathbb{P}\left[Y_{\mathbf v}\right]\ge 1- \exp{(-\xi^2p^{2r}n)} = 1- o_n(1).
\]
Suppose $Y_{\mathbf v}$ holds.
Then if \(\mathbf{v} \in I_{\mathcal{P},F}^{\mu^*}(H)\), then \(\mathbf{v} \in I_{\mathcal{P}',F}^{\mu^*/2}(H')\); if \(\mathbf{v} \notin I_{\mathcal{P},F}^{\mu^*/4}(H)\), then \(\mathbf{v} \notin I_{\mathcal{P}',F}^{\mu^*/2}(H')\).
Therefore, if $Y_{\mathbf v}$ holds for all \(r\)-vectors $\mathbf{v}$, then both $E_1$ and $E_2$ hold.
By union bound, we obtain \(\mathbb{P}[E_1\wedge E_2] \ge \mathbb{P} [Y_{\mathbf v} \text{ for all }\mathbf{v}] = 1- o_n(1)\).
\end{proof}

\subsection{Proof of Lemma~\ref{lem:lattice_prob}}

In this subsection, we give a proof of Lemma~\ref{lem:lattice_prob} using Lemma \ref{lem:lattice_points}, a result of Gauss concerning the intersection of a large convex body with a lattice of full rank.

\begin{proof}[Proof of Lemma~\ref{lem:lattice_prob}]
Set $1/n_0\ll1/C\ll1/R\ll \varepsilon, c,1/r,1/q,1/d$ and $C':=2C^2$.
Let $\mathcal{P}=\{V_0,V_1, \dots, V_d\}$ be a partition of $V$ where $n_i:=|V_i|\ge cn$ for each $i\in[d]$.
Let \(U:=V_p\) where $p(1-p)\ge C'n^{-1/2}$.
For each $i\in [d]$, let $U_i:=U\cap V_i$.
Let $X_v$ be the indicator variable of $v\in U$ for each $v\in V$. 
Clearly, $\{X_v\}_{v\in V}$ are independent random variables.


We choose a constant $\lambda\in [0,1)$ such that $pn_i+Cn^{1/2}+\lambda\notin \mathbb{Z}$ and $pn_i-Cn^{1/2}+\lambda\notin \mathbb{Z}$ for each $i\in[d]$.
Let $W:=\prod_{i=1}^d \Big((\,pn_i-Cn^{1/2}+\lambda,\, pn_i+Cn^{1/2}+\lambda\,)\cap \mathbb Z\Big)\subseteq \mathbb Z^d$. 
Without loss of generality, we may assume that \(R \mid 2Cn^{1/2}\) and set $m:=2Cn^{1/2}/R$.
Otherwise, replace \(2Cn^{1/2}\) by a number \(n'\) satisfying \(|n' -2Cn^{1/2}| \le R,\) which only changes the boundary of \(W\) and affects probability estimate by at most \(o_n(1)\). 
For each \(i\in[d]\) and \(s\in[m]\), define \[I_{i,s}:=(\,pn_i-Cn^{1/2}+(s-1)R+\lambda,\; pn_i-Cn^{1/2}+sR+\lambda).\] 
For each \(\mathbf j=(j_1,\dots,j_d)\in [m]^d\), let
$\tilde{B_{\mathbf j}}:=\prod_{i=1}^d I_{i,j_i}$ and $B_{\mathbf j}:=\tilde{B_{\mathbf j}}\cap \mathbb Z^d$.
Then the sets \(B_{\mathbf j}\) are pairwise disjoint and $\bigcup_{\mathbf j\in [m]^d} B_{\mathbf j}=W$, and $|B_{\mathbf j}|=\operatorname{mes}(\tilde B_{\mathbf j})=R^d$.

It is easy to see that $|U_i|=\sum_{v\in V_i}X_v$ and $\mathbb{E}[|U_i|]=pn_i$. 
By Lemma \ref{lem:Hoeffding}, we have
\[
 \mathbb{P}[|U_i|-pn_i\geq (Cn^{1/2}+\lambda)]\leq2\exp\left(-C^2\right),
\]
and 
\[
 \mathbb{P}[|U_i|-pn_i\leq (-Cn^{1/2}+\lambda)]\leq2\exp\left(-C^2\right),
\]
Taking a union bound over all $d$ parts, we obtain that 
$\mathbf{i}_{\mathcal P}(U)=(|U_1|,\dots, |U_d|)\in W$ holds with probability $1-O\left(\exp{\left(-C^2\right)}\right)$.

Next we claim that $\textbf{i}_{\mathcal{P}}(U)$ is almost uniformly distributed in $B_{\mathbf j}$ for each $\mathbf j\in[m]^d$.
Indeed, for $\mathbf{j} \in[m]^d$ and any $\textbf{x}=(x_1,\ldots,x_d),\textbf{y}=(y_1,\ldots,y_d)\in B_{\mathbf j}$, we have 
\[\frac{\mathbb{P}[\textbf{i}_{\mathcal{P}}(U)=\textbf{x}]}{\mathbb{P}[\textbf{i}_{\mathcal{P}}(U)=\textbf{y}]}=\prod_{i=1}^{d}\frac{\mathbb{P}[|U_i|=x_i]}{\mathbb{P}[|U_i|=y_i]}
=\prod_{i=1}^{d}\frac{\binom{n_i}{x_i}p^{x_i}(1-p)^{n_i-x_i}}{\binom{n_i}{y_i}p^{y_i}(1-p)^{n_i-y_i}}.\]
For $i\in [d]$, let $f(x)=\binom{n_i}{x}p^x(1-p)^{n_i-x}$.
For $x\in I_{i,s}$ for some $s\in [m]$, we have $\frac{f(x)}{f(x-1)}=\frac{p(n_i-x+1)}{(1-p)x}\le1+\frac{Cn^{1/2}+\lambda}{p(1-p)cn-(1-p)(Cn^{1/2}+\lambda)}=1+O(C^{-1})$ since $p(1-p)\ge C'n^{-1/2}$ and a similar calculation shows $\frac{f(x)}{f(x-1)}\ge 1-O(C^{-1})$.
Therefore,
$\frac{f(x_i)}{f(y_i)}= (1\pm O(C^{-1}))^R = 1 \pm O(RC^{-1})$,
as $|x_i-y_i|\le R$. 
Combining for all $i\in [d]$, we obtain
\[\frac{\mathbb{P}[\textbf{i}_{\mathcal{P}}(U)=\textbf{x}]}{\mathbb{P}[\textbf{i}_{\mathcal{P}}(U)=\textbf{y}]}=1\pm O(RC^{-1}).\] 
Thus, 
for any $\textbf{x}=(x_1,\ldots,x_d)\in B_{\mathbf j}$, we get
\[
\mathbb{P}[\textbf{i}_{\mathcal{P}}(U)=\textbf{x}\mid \textbf{i}_{\mathcal{P}}(U)\in B_{\mathbf j}]=\left(1\pm O(RC^{-1})\right)\dfrac{1}{|B_{\mathbf j}|}.
\]
Let $L:= L_{\mathcal{P},F}^{\mu_0}(H)$ and $\mathbf u_1=(1,0,\ldots,0)\in \mathbb{Z}^d$.
It is easy to see that $L_{\max}^d$ has exactly $r$ cosets $i\mathbf u_1+L_{\max}^d$ for $i=0,\dots, r-1$ in $\mathbb{Z}^d$, which yields \(|\mathbb{Z}^d/L_{\max}^{d}|=r\). 
Moreover, since $|Q(\mathcal{P},L)|=|L_{\max}^{d}/L|= q$, we have  
\(|\mathbb{Z}^d/L|=|\mathbb{Z}^d/L_{\max}^d|\cdot|L_{\max}^{d}/L|= rq\).
By Proposition \ref{prop:fullrank}, \(L\) is a full-rank sublattice of \(\mathbb Z^d\). In particular, the volume of a fundamental parallelepiped of \(L\) is \(\operatorname{covol}(L)=|\mathbb{Z}^d/L|=rq\). 
{For every coset \(\mathbf v+L\in \mathbb{Z}^d/L\), set \( L_{\mathbf v}:=\mathbf v+L \).
Then \(L_{\mathbf v}\) is a translate of \(L\), and  $\operatorname{covol}(L_{\mathbf v})=\operatorname{covol}(L)=rq$.} 
Note that each $\tilde{B_{\mathbf j}}$ where $\mathbf j\in[m]^d$ is a translation of $R\cdot (0,1)^d$.
Therefore, for every \(\mathbf j\in[m]^d\) and every \(L_{\mathbf{v}}\in \mathbb{Z}^d/L\),  using Lemma~\ref{lem:lattice_points}, we obtain that 
\[
|\tilde{B_{\mathbf j}}\cap L_{\mathbf{v}}|=\left(1\pm O(R^{-1})\right)\dfrac{\operatorname{mes}(\tilde{B_{\mathbf j}})}{rq}.
\] 
Moreover, we have $|B_{\mathbf j}|=\operatorname{mes}(\tilde B_{\mathbf j})=R^d$. 
As $B_{\mathbf j}=\tilde{B_{\mathbf j}}\cap \mathbb Z^d$, 
\[
|B_{\mathbf j}\cap L_{\mathbf{v}}|=|\tilde{B_{\mathbf j}}\cap L_{\mathbf{v}}|= \left(1\pm O\left(R^{-1}\right)\right)\dfrac{|B_{\mathbf j}|}{rq}.
\]
Combining the estimates, we obtain that for each $\mathbf j\in[m]^d$,
\begin{align*}
\mathbb{P}[\textbf{i}_{\mathcal{P}}(U)\in L_{\mathbf{v}}|\textbf{i}_{\mathcal{P}}(U)\in B_{\mathbf j}]&=\sum_{\textbf{x}\in B_{\mathbf j}\cap L_{\mathbf{v}}}\mathbb{P}[\textbf{i}_{\mathcal{P}}(U)=\textbf{x}|\textbf{i}_{\mathcal{P}}(U)\in B_{\mathbf j}]\\
&=(1\pm O({R}^{-1}))\dfrac{|B_{\mathbf j}|}{rq}\cdot\left(1\pm O(RC^{-1})\right)\dfrac{1}{|B_{\mathbf j}|}=\frac{1}{rq} +O({R}^{-1}),
\end{align*}
as $1/C\ll1/R$.
%
Recalling $W=\bigcup_{\mathbf j\in [m]^d} B_{\mathbf j}$, we have
\begin{align*}
\mathbb{P}[\textbf{i}_{\mathcal{P}}(U)\in L_{\mathbf{v}}\cap W]&=\sum_{\mathbf j\in[m]^d}\mathbb{P}[\textbf{i}_{\mathcal{P}}(U)\in B_{\mathbf j}]\cdot\mathbb{P}[\textbf{i}_{\mathcal{P}}(U)\in L_{\mathbf{v}}|\textbf{i}_{\mathcal{P}}(U)\in B_{\mathbf j}]\\
&=\sum_{\mathbf j\in[m]^d}\mathbb{P}[\textbf{i}_{\mathcal{P}}(U)\in B_{\mathbf j}]\left(\frac{1}{rq}+O({R}^{-1})\right)=\mathbb{P}[\textbf{i}_{\mathcal{P}}(U)\in W]\cdot\left(\frac{1}{rq}+O({R}^{-1})\right).
\end{align*}
As $\mathbb{P}[\textbf{i}_{\mathcal{P}}(U)\in W]=1-O\left(\exp{\left(-C^2\right)}\right)$, we get $\mathbb{P}[\textbf{i}_{\mathcal{P}}(U)\in L_{\mathbf{v}}\cap W]=\frac{1}{rq}+O({R}^{-1})$.
Finally,
\[
\mathbb{P}[\textbf{i}_{\mathcal{P}}(U)\in  L_{\mathbf{v}}]
= \mathbb{P}[\textbf{i}_{\mathcal{P}}(U)\in L_{\mathbf{v}}\setminus W]+\mathbb{P}[\textbf{i}_{\mathcal{P}}(U)\in L_{\mathbf{v}}\cap W]=\frac{1}{rq}+ O({R}^{-1}). 
\]
Since $1/R\ll \varepsilon$, we obtain $\mathbb{P}[\textbf{i}_{\mathcal{P}}(U)\in  L_{\mathbf{v}}]=\frac{1}{rq}\pm\varepsilon$.\qedhere
\end{proof}

\section{Sharpness of the winning probability}
\label{sec:win}
We now show that the winning probability in Theorem \ref{thm:unbalancedfactor} -- \ref{thm:perfectmatching} are asymptotically best possible.
Note that because all these theorems are derived by Theorem \ref{thm:main_structural_theorem_ell_degree}, it suffices to show that there exist infinitely many $k$-graphs satisfying the \textit{in particular} part of the theorem.

\medskip
\noindent
\textbf{For Theorem \ref{thm:unbalancedfactor}.}
Let $F$ be a graph. If $F$ is balanced, then take any graph $H$ with $\delta(H)\ge (1-1/\chi(F)+o(1))n$ and in our proof of Theorem \ref{thm:unbalancedfactor} we indeed take $q=1$.
Then by Chernoff's bound, $H[p]$ inherits the minimum degree condition of $H$ and thus has an $F$-factor if and only if its order is divisible by $r=|V(F)|$, which occurs with probability $1/r\pm o(1)$.

If $F$ is unbalanced, then take $H$ to be balanced complete $\chi(F)$-partite graphs, which satisfies that $\delta(H)\ge (1-1/\chi(F))n \ge (1-1/\chi_{cr}(F)+o(1))n$.
Let $\mathcal P$ be the natural partition of $V(H)$ and note that it satisfies all assumptions of Theorem \ref{thm:main_structural_theorem_ell_degree}.
Moreover, the index vector of every copy of \(F\) in \(H\) is \(\mu\)-robust for some small \(\mu>0\). 

\medskip
\noindent
\textbf{For Theorems \ref{thm:ell_degree} and \ref{thm:perfectmatching}.}
For the other two theorems, we use the following construction.
For Theorem \ref{thm:ell_degree}, set \(d:=\lceil\frac{1}{c^*_{k,\ell}+\gamma}\rceil-1\), so that \(1/d>c^*_{k,\ell}+\gamma\). For Theorem \ref{thm:perfectmatching}, set \(d:=s-1\) and take \(0<\gamma<1/(s(s-1))\). 

Partition the vertex set into $d$ parts $\mathcal P=\{V_1,\dots,V_d\}$ whose sizes differ by at most one. Define a $k$-graph $H$ on this vertex set by declaring a $k$-set $e$ to be an edge if and only if $\sum_{i=1}^{d}(i-1)|e\cap V_i|\equiv 0\pmod d$. When $d=1$, $H$ is the complete $k$-graph.
The key property is that for every $(k-1)$-set $S\subseteq V(H)$, there is a unique $j\in[d]$ such that $(j-1)+\sum_{i=1}^{d}(i-1)|S\cap V_i|\equiv 0\pmod d$. Consequently, $N_H(S)=V_j\setminus S$, and hence $\delta_{k-1}(H)\ge \lfloor n/d\rfloor-(k-1)\ge n/d-k$. 
For Theorem \ref{thm:perfectmatching}, since \(d=s-1\) and $1/(s-1)-1/s=1/(s(s-1))>\gamma$, we obtain $\delta_{k-1}(H)\ge (1/s+\gamma)n$ for all sufficiently large $n$. 
More generally, for every \(\ell\)-set \(T\subseteq V(H)\), we have \(\deg_H(T)\ge \frac{1}{k-\ell}\binom{n-\ell}{k-\ell-1}\cdot\delta_{k-1}(H)=\left(\frac{1}{d}-o(1)\right)\binom{n-\ell}{k-\ell}\). 
Thus, with the above choice of \(d\) for Theorem~\ref{thm:ell_degree}, we have \(\delta_{\ell}(H)\ge(c^*_{k,\ell}+\gamma)\binom{n-\ell}{k-\ell}\). 

Note that by definition, given a $k$-vector $\vec{v}\in \mathbb{Z}^d$, the  $k$-tuples with index vector $\vec{v}$ are either all edges of $H$ or all non-edges of $H$.
Since the parts of \(\mathcal P\) have sizes \(n/d\pm1\), this implies that the index vector of every edge of \(H\) is \(\mu\)-robust for some sufficiently small \(\mu=\mu(k,d)>0\).

Finally, it is easy to see that the constructions above are insensitive to small perturbations of part size up to $o(n)$ (in fact, one may also delete a small number of edges in each neighborhood), which generates infinitely many examples.

\section{Acknowledgement}
We would like to sincerely thank Peter Keevash for pointing out Lemma \ref{lem:lattice_points} from \cite{tao2006additive} to us.

\bibliographystyle{plain}
\bibliography{references}

@article{han2020complexity,
  title={{The complexity of perfect matchings and packings in dense hypergraphs}},
  author={Han, J. and Treglown, A.},
  journal={J. Combin. Theory Ser. B},
  volume={141},
  pages={72--104},
  year={2020},
  publisher={Elsevier}
}

@incollection{mcdiarmid,
    AUTHOR = {McDiarmid, C.},
     TITLE = {On the method of bounded differences},
 BOOKTITLE = {Surveys in combinatorics, 1989 ({N}orwich, 1989)},
    SERIES = {London Math. Soc. Lecture Note Ser.},
    VOLUME = {141},
     PAGES = {148--188},
 PUBLISHER = {Cambridge Univ. Press, Cambridge},
      YEAR = {1989},
      ISBN = {0-521-37823-0},
   MRCLASS = {05C80 (60E15 60F10 60G42)},
  MRNUMBER = {1036755},
MRREVIEWER = {Alan\ M.\ Frieze},
}

@book{tao2006additive,
    AUTHOR = {Tao, T. and Vu, V.},
     TITLE = {Additive combinatorics},
    SERIES = {Cambridge Studies in Advanced Mathematics},
    VOLUME = {105},
 PUBLISHER = {Cambridge University Press, Cambridge},
      YEAR = {2006},
     PAGES = {xviii+512},
      ISBN = {978-0-521-85386-6; 0-521-85386-9},
}

@article {MR144363,
    AUTHOR = {Hoeffding, W.},
     TITLE = {Probability inequalities for sums of bounded random variables},
   JOURNAL = {J. Amer. Statist. Assoc.},
    VOLUME = {58},
      YEAR = {1963},
     PAGES = {13--30},
}

@incollection {MR297607,
AUTHOR = {Hajnal, A. and Szemer\'edi, E.},
TITLE = {Proof of a conjecture of {P}. {Erd\H{o}s}},
BOOKTITLE = {Combinatorial theory and its applications, {I}-{III} ({P}roc.
              {C}olloq., {B}alatonf\"ured, 1969)},
    SERIES = {Colloq. Math. Soc. J\'anos Bolyai},
    VOLUME = {4},
     PAGES = {601--623},
 PUBLISHER = {North-Holland, Amsterdam-London},
      YEAR = {1970},
   MRCLASS = {05C99},
  MRNUMBER = {297607},
MRREVIEWER = {J.\ W.\ Moon},
}

@article {MR4935989,
    AUTHOR = {Dragani\'c, N. and Keevash, P. and M\"uyesser, A.},
     TITLE = {Cyclic subsets in regular {D}irac graphs},
   JOURNAL = {Int. Math. Res. Not.},
      YEAR = {2025},
    NUMBER = {14},
}

@article{MR1684620,
    AUTHOR = {Erd\H{o}s, P.},
     TITLE = {A selection of problems and results in combinatorics},
      NOTE = {Recent trends in combinatorics (M\'atrah\'aza, 1995)},
   JOURNAL = {Combin. Probab. Comput.},
    VOLUME = {8},
      YEAR = {1999},
    NUMBER = {1-2},
     PAGES = {1--6},
      ISSN = {0963-5483,1469-2163},
   MRCLASS = {05-XX},
  MRNUMBER = {1684620},
MRREVIEWER = {Andr\'as\ Gy\'arf\'as},
       DOI = {10.1017/S0963548398003496},
       URL = {https://doi.org/10.1017/S0963548398003496},
}

@article {MR5041387,
    AUTHOR = {Hunter, Z. and Liu, T. and Milojevi\'c, A. and
              Sudakov, B.},
     TITLE = {Cyclic subsets of tournaments},
   JOURNAL = {Random Structures Algorithms},
    VOLUME = {68},
      YEAR = {2026},
    NUMBER = {2},
     PAGES = {Paper No. e70056},
      ISSN = {1042-9832,1098-2418},
   MRCLASS = {05C45 (05C20)},
  MRNUMBER = {5041387},
       DOI = {10.1002/rsa.70056},
       URL = {https://doi.org/10.1002/rsa.70056},
}

@article {Yang2025fator,
    AUTHOR = {Sun, W. and Wei, S. and Yang, D.},
     TITLE = {Clique factors in random samplings of regular graphs},
   JOURNAL = {arXiv:2512.20287v1},
      YEAR = {2025},
}

@article {MR707416,
    AUTHOR = {Kirkpatrick, D. G. and Hell, P.},
     TITLE = {On the complexity of general graph factor problems},
   JOURNAL = {SIAM J. Comput.},
    VOLUME = {12},
      YEAR = {1983},
    NUMBER = {3},
     PAGES = {601--609},
      ISSN = {0097-5397},
   MRCLASS = {68R10 (05C70 90B10)},
  MRNUMBER = {707416},
       DOI = {10.1137/0212040},
       URL = {https://doi.org/10.1137/0212040},
}

@article{KSS2001factor,
    AUTHOR = {Koml\'os, J. and S\'ark\"ozy, G. N. and Szemer\'edi,
              E.},
     TITLE = {Proof of the {A}lon-{Y}uster conjecture},
      NOTE = {Combinatorics (Prague, 1998)},
   JOURNAL = {Discrete Math.},
    VOLUME = {235},
      YEAR = {2001},
    NUMBER = {1-3},
     PAGES = {255--269},
      ISSN = {0012-365X,1872-681X},
   MRCLASS = {05C15},
  MRNUMBER = {1829855},
       DOI = {10.1016/S0012-365X(00)00279-X},
       URL = {https://doi.org/10.1016/S0012-365X(00)00279-X},
}

@article {alonyuster,
    AUTHOR = {Alon, N. and Yuster, R.},
     TITLE = {{$H$}-factors in dense graphs},
   JOURNAL = {J. Combin. Theory Ser. B},
    VOLUME = {66},
      YEAR = {1996},
    NUMBER = {2},
     PAGES = {269--282},
}

@article {kuhnothus,
    AUTHOR = {K\"uhn, D. and Osthus, D.},
     TITLE = {The minimum degree threshold for perfect graph packings},
   JOURNAL = {Combinatorica},
    VOLUME = {29},
      YEAR = {2009},
    NUMBER = {1},
     PAGES = {65--107},
}

@article{han2017decision,
  title={{Decision problem for perfect matchings in dense \(k\)-uniform hypergraphs}},
  author={Han, J.},
  journal={Trans. Amer. Math. Soc.},
  volume={369},
  number={7},
  pages={5197--5218},
  year={2017}
}

@article{shokoufandeh2003proof,
  title={{Proof of a tiling conjecture of Koml{\'o}s}},
  author={Shokoufandeh, A. and Zhao, Y.},
  journal={Random Structures Algorithms},
  volume={23},
  number={2},
  pages={180--205},
  year={2003},
  publisher={Wiley Online Library}
}

@article {RSS2009perfect,
    AUTHOR = {R\"odl, V. and Ruci\'nski, A. and Szemer\'edi,
              E.},
     TITLE = {Perfect matchings in large uniform hypergraphs with large
              minimum collective degree},
   JOURNAL = {J. Combin. Theory Ser. A},
    VOLUME = {116},
      YEAR = {2009},
    NUMBER = {3},
     PAGES = {613--636},
}

@article {komlos2000almost,
    AUTHOR = {Koml\'os, J.},
     TITLE = {Tiling {T}ur\'an theorems},
   JOURNAL = {Combinatorica},
    VOLUME = {20},
      YEAR = {2000},
    NUMBER = {2},
     PAGES = {203--218},
}

@article {Han2015near,
    AUTHOR = {Han, J.},
     TITLE = {Near perfect matchings in {$k$}-uniform hypergraphs},
   JOURNAL = {Combin. Probab. Comput.},
    VOLUME = {24},
      YEAR = {2015},
    NUMBER = {5},
     PAGES = {723--732},
}

@article {KKM2015matchingpoly,
    AUTHOR = {Keevash, P. and Knox, F. and Mycroft, R.},
     TITLE = {Polynomial-time perfect matchings in dense hypergraphs},
   JOURNAL = {Adv. Math.},
    VOLUME = {269},
      YEAR = {2015},
     PAGES = {265--334},
}

@article {KRZ2010matching,
    AUTHOR = {Karpi\'nski, M. and Ruci\'nski, A. and Szyma\'nska,
              E.},
     TITLE = {Computational complexity of the perfect matching problem in
              hypergraphs with subcritical density},
   JOURNAL = {Internat. J. Found. Comput. Sci.},
    VOLUME = {21},
      YEAR = {2010},
    NUMBER = {6},
     PAGES = {905--924},
}

@article {sudakov2008robust,
    AUTHOR = {Sudakov, B. and Vu, V. },
     TITLE = {Local resilience of graphs},
   JOURNAL = {Random Structures Algorithms},
    VOLUME = {33},
      YEAR = {2008},
    NUMBER = {4},
     PAGES = {409--433},
}

@inproceedings{han2026perfect,
  title={Perfect Matchings in Random Sparsifications of Dense Hypergraphs},
  author={Han, J. and Zhao, J.},
  booktitle={Proceedings of the 2026 Annual ACM-SIAM Symposium on Discrete Algorithms (SODA)},
  pages={2430--2454},
  year={2026},
  organization={SIAM}
}

@article{gan2025keevash,
  title={On the {K}eevash-{K}nox-{M}ycroft {C}onjecture},
  author={Gan, L. and Han, J.},
  journal={J. Combin. Theory Ser. B},
  volume={174},
  pages={214--242},
  year={2025},
  publisher={Elsevier}
}

@article{ChangGeHanWang2022,
  author  = {Chang, Y. and Ge, H. and Han, J. and Wang, G.},
  title   = {Matching of Given Sizes in Hypergraphs},
  journal = {SIAM J. Discrete Math.},
  volume  = {36},
  number  = {3},
  pages   = {2323--2338},
  year    = {2022}
}

@article{liebenau2023asymptotic,
  title={{Asymptotic enumeration of graphs by degree sequence, and the degree sequence of a random graph}},
  author={Liebenau, A. and Wormald, N.},
  journal={J. Eur. Math. Soc.},
  volume={26},
  number={1},
  pages={1--40},
  year={2023}
}

@article{alon2012large,
  title={{Large matchings in uniform hypergraphs and the conjectures of Erd{\H{o}}s and Samuels}},
  author={Alon, N. and Frankl, P. and Huang, H. and R{\"o}dl, V. and Ruci{\'n}ski, A. and Sudakov, B.},
  journal={J. Combin. Theory Ser. A},
  volume={119},
  number={6},
  pages={1200--1215},
  year={2012},
  publisher={Elsevier}
}

@article{frankl2022erdHos,
  title={{The Erd{\H{o}}s matching conjecture and concentration inequalities}},
  author={Frankl, P. and Kupavskii, A.},
  journal={J. Combin. Theory Ser. B},
  volume={157},
  pages={366--400},
  year={2022},
  publisher={Elsevier}
}

@article{fu2026sharp,
  title   = {Sharp small-deviation inequalities for sums of independent nonnegative random variables},
  author  = {Fu, W. and Han, Y. and Wang, G. and Yan, J. and Zhang, P. and Zhou, Z.},
  journal = {arXiv:2607.23980},
  year    = {2026}
}

@article{ferber2019uniformity,
  title={Uniformity-independent minimum degree conditions for perfect matchings in hypergraphs},
  author={Ferber, A. and Jain, V.},
  journal={arXiv:1903.12207},
  year={2019}
}

@article{LiuNiuWangYan2026,
  author        = {Liu, H. and Niu, M. and Wang, L. and Yan, Z.},
  title         = {Tight Staircase Bounds for Cyclic Subsets below {D}irac's Threshold},
  journal       = {arXiv:2607.06551},
  year          = {2026},
  eprint        = {2607.06551},
}

\end{document}